\documentclass[11pt,a4paper]{article}
\usepackage[a4paper]{geometry}
\usepackage[dvips]{graphicx}
\usepackage{amsthm}
\usepackage{amsmath}
\usepackage{mathrsfs}
\usepackage{amssymb}
\usepackage{verbatim}
\usepackage{algorithmic}
\usepackage{algorithm}
\usepackage{color}
\usepackage{enumerate}
\usepackage{appendix}
\usepackage{marginnote}

\usepackage{hyperref} \hypersetup{
pdftitle={Shown in AR File Information}, pdfstartview=FitH, % Fit the page horizontally bookmarks=true,
} % more options can be found in % TEXMF/doc/latex/hyperref/manual.pdf

\newcommand\mydef{\mathrel{\overset{\makebox[0pt]{\mbox{\normalfont\tiny\sffamily def}}}{=}}}

\newcommand{\Ll}{\mathcal{L}}
\newcommand{\Ga}{\mathcal{G}}

\newcommand{\sos}{\text{\sc{SoS}}}

\newcommand{\F}{\mathcal{F}}

\newcommand{\Cc}{\mathcal{C}}

\newcommand{\Ideal}[1]{{\textbf{I}}\left( #1 \right)}
\newcommand{\GIdeal}[1]{\left\langle #1 \right\rangle}

\newcommand{\CSP}{\textsc{CSP}}
\newcommand{\IMP}{\textsc{IMP}}
\newcommand{\Pol}{\textsf{Pol}}

\newcommand{\Max}{\textsf{Max}}
\newcommand{\Min}{\textsf{Min}}
\newcommand{\Majority}{\textsf{Majority}}
\newcommand{\Minority}{\textsf{Minority}}

\newcommand{\ThB}{\textsc{TH}}
\newcommand{\Variety}[1]{{\textbf{V}}\left( #1 \right)}
\newcommand{\spn}[1]{\left\langle #1 \right\rangle}

\newcommand{\x}{\mathbf{x}}
\newcommand{\I}{\emph{\texttt{I}}}
\newcommand{\Zz}{\mathbb{Z}}
\newcommand{\N}{\mathbb{N}}

\newcommand{\multideg}{\textnormal{multideg}}
\newcommand{\LM}{\textnormal{LM}}
\newcommand{\LT}{\textnormal{LT}}
\newcommand{\LC}{\textnormal{LC}}
\newcommand{\LCM}{\textnormal{lcm}}
\newcommand{\GB}{\text{Gr\"{o}bner }}
\newcommand{\reduce}[2]{\overline{#1}^{#2}}

\newcommand{\Field}{\mathbb{F}}
\newcommand{\Real}{\mathbb{R}}

\newcommand{\T}{\mathcal{T}}

\newtheorem{theorem}{Theorem}[section]

\newtheorem{lemma}[theorem]{Lemma}
\newtheorem{corollary}[theorem]{Corollary}

\newtheorem{proposition}[theorem]{Proposition}
\newtheorem{remark}{Remark}[section]

\newtheorem{question}[theorem]{Research Questions}
\newtheorem{example}{Example}[section]
\newtheorem{definition}{Definition}[section]

\title{The Complexity of the Ideal Membership Problem for Constrained Problems Over the Boolean Domain}
\author{ Monaldo Mastrolilli\\
\small{IDSIA, Lugano, Switzerland.}\\
\small{ monaldo@idsia.ch}}

\date{}

\begin{document}
\maketitle

\begin{abstract}
Given an ideal $\I$ and a polynomial $f$ the {Ideal Membership Problem} is to test if $f\in \I$. This problem is a fundamental algorithmic problem with important applications and notoriously intractable.

We study the complexity of the {Ideal Membership Problem} for combinatorial ideals that arise from constrained problems over the Boolean domain. As our main result, we identify the precise borderline of tractability. Our result generalizes Schaefer's dichotomy theorem [STOC, 1978] which classifies all Constraint Satisfaction Problems over the Boolean domain to be either in P or NP-hard.

This paper is motivated by the pursuit of understanding the recently raised issue of bit complexity of Sum-of-Squares proofs [O'Donnell, ITCS, 2017]. Raghavendra and Weitz [ICALP, 2017] show
how the Ideal Membership Problem tractability for combinatorial ideals implies bounded coefficients in Sum-of-Squares proofs.
\end{abstract}

%%%%%%%%%%%%%%%%%%%%%%%
\section{Introduction}
%%%%%%%%%%%%%%%%%%%%%%%

The polynomial \textsc{Ideal Membership Problem} (\IMP) is the following computational task.
Let $\Field[x_1, \ldots, x_n]$ be the ring of polynomials over a field $\Field$ and indeterminates $x_1,\ldots, x_n$ (for the applications of this paper $\Field=\Real$).
Given $f_0,f_1,\ldots,f_r\in \Field[x_1, \ldots, x_n]$ we want to decide if $f_0\in \I= \GIdeal{f_1,\ldots, f_r}$, where $\I$ is the ideal generated by $F=\{f_1,\ldots , f_r\}$.
This problem was first studied by Hilbert~\cite{Hilbert1893}, and it is a fundamental algorithmic problem with important applications in solving polynomial systems and polynomial identity testing (see e.g.~\cite{Cox:2015}).
In general, however, $\IMP$ is notoriously intractable. The results of Mayr and Meyer show that it is EXPSPACE-complete~\cite{mayr89,MAYR82}.
%If we restrict to $f_0,f_1,\ldots,f_r$ of a special form, often dramatic improvements are possible: for example if $\I$ is zero-dimensional then the membership can be decided in single-exponential time \cite{DICKENSTEIN1991}.
See~\cite{Mayr2017} and the references therein for a recent survey.

The $\IMP$ is efficiently solvable if there exist ``low-degree'' proofs of membership for the ideal $\I$ generated by $F$, namely $f_0=\sum_{f \in F} q_f \cdot f$ for some polynomials $\{q_f\in \Field[x_1, \ldots, x_n]\mid  f\in F\}$ with $q_f\cdot f$ of degree $\leq k\cdot d$, where $d$ is the degree of $f_0$ and $k$ is ``small''. If the latter holds for any given $f_0\in \I$ of degree at most $d$ then, following the notation in \cite{Weitz17}, we say that $F$ is \emph{$k$-effective}.
The Effective Nullstellensatz \cite{Hermann1926} tells us that we can take $k\leq d^{2^{O(|F|)}}$, which is not a very useful bound in practice. However, this bound is unavoidable in general because of the EXPSPACE-hardness.

If we restrict to $f_0,f_1,\ldots,f_r$ of a special form, often dramatic improvements are possible: for example if $\I$ 
is zero-dimensional, namely when the polynomial system $f_1=\dots=f_r=0$  has only a finite number of solutions, then the membership can be decided in single-exponential time~\cite{DICKENSTEIN1991}.
Moreover, the polynomial ideals that arise from combinatorial optimization problems frequently have special nice properties. For instance, these ideals are often Boolean and therefore zero-dimensional and radical.
For combinatorial problems the $\IMP$ has been studied mostly in the context of \emph{lower bounds}, see e.g. \cite{BeameIKPP94,BussP98,Grigoriev98}. In these works
a set of polynomials forming a derivation is called a \emph{Polynomial Calculus} or \emph{Nullstellensatz} proof and
the problem is referred to as the degree of Nullstellensatz proofs of membership for the input polynomial $f_0$. Clegg, Edmonds and Impagliazzo~\cite{CleggEI96} use polynomials to represent finite-domain constraints and discuss a propositional proof system based on a bounded degree version of Buchberger's algorithm~\cite{BuchbergerThesis}, called \GB proof system, for finding proofs of unsatisfiability, which corresponds to the very special case of the $\IMP$ with $f_0=1$.
%In~\cite{JeffersonJGD13} some initial findings on using \GB basis techniques for solving constraint problems are presented. The authors raised the question to determine for which classes of constraint problems \GB techniques can be efficiently applied.
%

Recently, Raghavendra and Weitz~\cite{RaghavendraW17,Weitz17} obtain \emph{upper bounds} on the required degree for Nullstellensatz proofs for several ideals arising from a number of combinatorial problems that are \emph{highly symmetric}, including Matching, TSP, and Balanced CSP.
However, their strategy is by no means universally applicable, and it had to be applied on a case-by-case basis.
Raghavendra and Weitz~\cite{RaghavendraW17,Weitz17} use the existence of low-degree Nullstellensatz proofs for combinatorial ideals to bound the bit complexity of the {Sum-of-Squares} relaxations/proof systems, as explained below.

The \emph{Sum-of-Squares} ($\sos$) proof system is a systematic and powerful approach to certifying polynomial inequalities. $\sos$ certificates can be shown to underlie a large number of algorithms in combinatorial optimization.
It has often been claimed in recent papers that one can compute a degree $d$ $\sos$ proof (if one exists) via the Ellipsoid algorithm in $n^{O(d)}$ time.
In a recent work, O'Donnell~\cite{ODonnell17} observed that this often repeated claim is far from true. O'Donnell gave an example of a polynomial system and a polynomial which had degree two proofs of non-negativity with coefficients requiring an exponential number of bits, causing the Ellipsoid algorithm to take exponential time.
On a positive note he
showed that a polynomial system whose \emph{only} constraints are the Boolean constraints
 $\{x_i^2=1 \mid i\in [n]\}$ always admit $\sos$ proofs with polynomial bit complexity and asked whether every polynomial system with Boolean constraints admits a small $\sos$ proof. This question is answered in the negative in~\cite{RaghavendraW17}, giving a counterexample and leaving open the question under which restrictions polynomial systems with Boolean constraints admit small $\sos$ proofs.
More in general O'Donnell~\cite{ODonnell17} raises the open problem to establish useful conditions under which ``small'' $\sos$ proof can be guaranteed automatically.

A first elegant approach to this question is due to Raghavendra and Weitz~\cite{RaghavendraW17} by providing a \emph{sufficient} condition on a polynomial system  that implies bounded coefficients in $\sos$ proofs.
In particular, let $Sol(\Cc)$ be the set of feasible solutions of a given Boolean combinatorial problem $\Cc$ and let $\I_\Cc$ be the vanishing ideal of set $Sol(\Cc)$. We will refer to $\I_\Cc$ as the \emph{combinatorial} ideal of $\Cc$. If a given combinatorial ideal generating set $\{f_1,\ldots,f_r\}$ (i.e. $\I_\Cc=\GIdeal{f_1,\ldots,f_r}$) is $k$-effective for constant $k=O(1)$, then they show that any polynomial $p$ that is non-negative on $Sol(\Cc)$ and that admits a degree $d$-$\sos\mod\{f_1,\ldots,f_r\}$ proof  of non-negativity~\footnote{Meaning that the non-negative polynomial $p$ over $Sol(\Cc)$ can be written as $p=\sigma+\sum_{i}f_i\cdot q_i$ where $\sigma$ is a sum of squares polynomial and $\sigma,(f_i\cdot q_i)\in \mathbb{R}[x]_{2d}$.}, is guaranteed to have a degree $k\cdot d$-$\sos\mod\{f_1,\ldots,f_r\}$ certificate with polynomial bit complexity.
So, as remarked in~\cite{RaghavendraW17},  ``the only non-trivial thing to verify is the efficiency of the polynomial calculus proof system''.
Actually, Weitz in his thesis~\cite{Weitz17} raised the following open question:
``Is there a criterion for combinatorial ideals that suffices to show that a set of polynomials admits $k$-effective derivations for constant $k$?'' and suggested to study problems without the strong symmetries discussed in his thesis and article (see the solution of the suggested starting problem from his thesis in Section~\ref{example-VC}).

Note that the sufficient criterion of Raghavendra and Weitz~\cite{RaghavendraW17} implies a somehow stronger approach: if we can efficiently compute a generating set $F=\{f_1,\ldots,f_r\}$ of $\I_C$ such that $F$ is $k$-effective with $k=O(1)$, then this gives a $k\cdot d$-$\sos\mod\{f_1,\ldots,f_r\}$  proof with polynomial bit complexity
for any polynomial that is nonnegative on $Sol(\Cc)$ and that admits a proof of non-negativity by a degree $d$-$\sos\mod\{f_1,\ldots,f_r\}$
certificate.
So the main open question with this approach is the following.
\begin{question}\label{Q:effective}
 Which restrictions on combinatorial problems can guarantee efficient computation of $O(1)$-effective generating sets?
\end{question}
The expert reader has probably realized that the efficient computation of effective generating sets  leads to the efficient construction of Theta Bodies SDP relaxations~\cite{GouveiaPT10}.
For a positive integer~$d$, the $d$-th Theta Body of an ideal $\I\in \Real[x]$ is
\begin{align*}
   \ThB_d(\I) & \mydef\{x\in \Real^n\mid \ell(x)\geq 0 \text{ for every linear }\ell \text{ that is } d\text{-}\sos\text{ mod }\I\},
\end{align*}
where a polynomial $f$ is $d$-$\sos$ mod $\I$ if there exists a finite set of polynomials $h_1,\ldots,h_t\in \Real[x]_d$ such that $f-\sum_{j=1}^{t}h_j^2\in \I$. The $d$-th Theta Body relaxation of $\I$ finds a certificate of non-negativity for $p$ which is a sum-of-squares polynomial $\sigma\in \Real[x]_{2d}$ together with a polynomial $g$ from the ideal $\I$ such that $p=\sigma+g$.

Theta Bodies are nice and elegant SDP relaxations which generalize
the Theta Body of a graph constructed by Lov\'{a}sz while studying the Shannon capacity of graphs~\cite{Lovasz79}. They are known to have several interesting properties~\cite{GouveiaPT10} and deep implications for approximation, for example they achieve the best approximation among all symmetric SDPs of a comparable size \cite{Weitz17}. %Moreover, for combinatorial problems they do not suffer of bit complexity issues due to the result by Raghavendra and Weitz~\cite{RaghavendraW17}.
However, to get our hands on this, we would need to be able to \emph{at least} solve the {\IMP} for combinatorial ideals up to a constant degree ({$\IMP_d$}).
For some problems $\IMP_d$ may be intractable, and so even trying to formulate the $d$-th Theta Body is intractable. Moreover, the $\IMP_d$ complexity is far from being well understood.
%
%Unfortunately, this problem is frequently intractable, and so even trying to formulate the Dth Theta Body is intractable
%
As a matter of fact there are only very few examples of efficiently constructible Theta Bodies relaxations.

%In this paper we identify restrictions on the combinatorial ideal $I_\Cc$ based on the so-called \emph{constraint language} that are \emph{necessary and sufficient} for applying Raghavendra and Weitz's criteria automatically and efficiently.

%Note that for the simplest case with \emph{only} Boolean constraints the $\IMP$ is straightforward since $\{x_i^2-x_i=0:i\in [n]\}$ is a {\GB basis} and therefore 1-effective.
%
For any given ideal $\I$, there is a particular kind of generating set $G=\{g_1,\ldots,g_t\}$ of the ideal~$\I$  that always admits 1-effective proofs. This set $G$ is known as {\GB basis}. More precisely, for testing the ideal $\I$ membership of a given degree-$d$ polynomial it is sufficient to compute
the set $G_d$ of polynomials with degree $\leq d$ of the reduced \GB basis $G$ for~$\I$ (assuming a grlex ordering).
This would also yield an efficient construction of the corresponding Theta Body relaxation and guarantees bounded coefficients in $\sos$ proofs.
Indeed, as shown in~\cite{GouveiaPT10}, the $d$-th Theta Body $\ThB_d(\I_\Cc)$ of a combinatorial ideal $\I_\Cc\in \Real[x]$ can be formulated as a projected spectrahedron which enables computations via Semi-Definite Programming (SDP). The SDP relaxation is derived by computing the so called $d$-th \emph{reduced moment matrix} of $\I_\Cc$ which can be obtained via \GB theory by computing the aforementioned set $G_d$ of polynomials with degree $\leq d$ of the reduced \GB basis $G$ for~$\I_\Cc$ (assuming a grlex ordering).

A {\GB basis} allows many important properties of the ideal and the associated algebraic variety to be deduced easily. \GB basis computation is one of the main practical tools for solving systems of polynomial equations. It can be seen as a multivariate, non-linear generalization of both Euclid's algorithm for computing polynomial greatest common divisors, and Gaussian elimination for linear systems.
Computational methods are an established tool in algebraic geometry and commutative algebra, the key element being the theory of \GB bases.
Buchberger~\cite{BuchbergerThesis} in his Ph.D. thesis (1965) introduced this important concept of a \GB basis and gave an algorithm for deciding ideal membership which is widely used today (see, e.g.~\cite{Cox:2015}).

The complexity of \GB bases has been the object of extensive studies (see e.g.~\cite{gbbib} and the references therein). It is well-known that in the worst-case, the complexity is doubly exponential in the number of variables, which is a consequence of the EXPSPACE-hardness of the $\IMP$ ~\cite{mayr89,MAYR82}.
These worst-case estimates have led to the unfortunately widespread belief that \GB
bases are not a useful tool beyond toy examples. However, it has been observed for a long
time that the actual behaviour of \GB bases implementations can be quite efficient.
This motivates an investigation
of the complexity of \GB basis algorithms for useful special classes of polynomial systems.

\paragraph{This paper:}
In this paper we consider vanishing ideals of feasible solutions that arise from Boolean combinatorial optimization problems.
The question of identifying restrictions to these problems which are sufficient to ensure the ideal membership tractability is important from both a practical and a theoretical viewpoint, and has an immediate application to $\sos$ proof complexity, as already widely remarked. Such restrictions may either consider the \emph{structure} of the constraints, namely which variables may be constrained by which other variables, or they may involve the \emph{nature} of the constraints, in other words, which combination of values are permitted for variables that are mutually constrained.

In this paper we take the second approach by restricting the so-called \emph{constraint language} (see Definition~\ref{def:constraint_language}), namely a set of relations that is used to form constraints. Each constraint language $\Gamma$ gives rise to a particular polynomial ideal membership problem, denoted $\IMP(\Gamma)$, and the goal is to describe the complexity of $\IMP(\Gamma)$ for all constraint languages~$\Gamma$.

This kind of restrictions on the constraint languages have been successfully applied to study the computational complexity classification (and other algorithmic properties) of the decision version of $\CSP$ over a fixed constraint language $\Gamma$ on a finite domain, denoted $\CSP(\Gamma)$ (see Section~\ref{sect:csp_intro}).
%This restricted framework is still broad enough to include many decision problems from the class NP, yet it is narrow enough to potentially allow for complete classifications of all such $\CSP$s.
This classification started with the classic dichotomy  result of Schaefer \cite{Schaefer78} for 0/1 $\CSP$s, and culminated with the recent papers by Bulatov~\cite{Bulatov17} and Zhuk~\cite{Zhuk17}, settling the long-standing Feder-Vardi dichotomy conjecture for finite domain $\CSP$s. We refer to~\cite{2017dfu7} for an excellent survey.

%\cite{Loera15}

%%%%%%%%%%%%%%%%%%%%%%%%%%%%%%%%%%%%%%%%%%%%%%%%%%%%%%%%%%%%%%%%%%%%%%%%%%%%%%%%%
\subsection{Main Results}
We begin with a formal definition of the problem.
Let $\Cc=(X,D,C)$ denote an instance of a given $\CSP(\Gamma)$,
where $X=\{x_1,\ldots,x_n\}$ is a set of $n$ variables, $D=\{0,1\}$ and $C$  is a set of constraints over $\Gamma$ with variables from $X$.
Let $Sol(\Cc)$ be the (possibly empty) set of satisfying assignments for $\Cc$, i.e.
the set of all mappings $\phi: X\rightarrow D$ satisfying all of the constraints from $C$ (see Section~\ref{sect:csp_intro} for additional details).

The \emph{combinatorial ideal} $\I_\Cc$ is the vanishing ideal of set $Sol(\Cc)$, namely it is the set of all polynomials that vanish on $Sol(\Cc)$: $f\in \I_\Cc$ if and only if $f(a_1,\ldots,a_n)=0$ for all $(a_1,\ldots,a_n)\in Sol(\Cc)$. Note that  $Sol(\Cc)=\Variety{\I_\Cc}$, namely the set of satisfying assignments corresponds to the variety of $\I_\Cc$  (see Section~\ref{sect:idealCSP} for additional details).
\begin{definition}\label{def:IMP}
 The {\emph{\textsc{Ideal Membership Problem}}} associated with language $\Gamma$ is the problem $\emph{\IMP}(\Gamma)$ in which
 the input consists of a polynomial $f\in \Field[X]$ and a $\emph{\CSP}(\Gamma)$ instance $\Cc=(X,D,C)$.
The goal is to decide %if $f\in \I_\Cc$.
 whether $f$ lies in the combinatorial ideal~$\I_\Cc$.
We use $\IMP_d(\Gamma)$ to denote $\IMP(\Gamma)$ when the input polynomial $f$ has degree at most $d$.
\end{definition}
Observe that $\IMP(\Gamma)$ belongs to the complexity class coNP.
Indeed, for any `no-instance', namely $f\not\in \I_\Cc$, there is a certificate which a polynomial-time algorithm can use to verify that $f\not\in \I_\Cc$. For example, such a certificate is given by an element of the corresponding variety $Sol(\Cc)=\Variety{\I_\Cc}$ that is not a zero of $f$. Note that if the instance $\Cc$ has no solution, i.e. $Sol(\Cc)=\Variety{\I_\Cc}=\emptyset$, then it is vacuously true that there is no element from  $\Variety{\I_\Cc}$ that is not a zero of $f$, whatever is $f$. %, for example when $f$ is the constant polynomial $f=1$.
It follows (see the Weak Nullstellensatz~\eqref{eq:weak_nstz}) that a given $\CSP(\Gamma)$ instance $\Cc$ has a solution if and only if $1\not\in \I_\Cc$. In other words, $\CSP(\Gamma)$ is equivalent to not-$\IMP_0(\Gamma)$.

In this paper we consider the question of identifying restrictions on the constraint language $\Gamma$ which ensure the $\IMP(\Gamma)$ tractability.
The complexity of $\IMP_d(\Gamma)$, for any $d\geq 0$, is obtained by arguing along the following lines.
\begin{itemize}
    \item[$(d=0)$] Schaefer's Dichotomy result \cite{Schaefer78} gives necessary and sufficient conditions on $\Gamma$ to ensure the tractability of $\CSP(\Gamma)$. 
    According to Schaefer's result (see Lemma~\ref{th:post} and Theorem~\ref{th:schaefer} below), the problem $\CSP(\Gamma)$ is polynomial-time tractable if the solution space of every relation in $\Gamma$ is closed under one of the following six operations (otherwise it is NP-complete): the ternary $\Majority$ operation, the ternary $\Minority$ operation, the $\Min$ operation, the $\Max$ operation, the constant operation $0$, the constant operation $1$.
    Since  $\CSP(\Gamma)$ is equivalent to not-$\IMP_0(\Gamma)$, it follows that Schaefer's Dichotomy result fully determines whether $\IMP_0(\Gamma)$ is polynomial-time solvable or coNP-complete. 
    %%%%%%%%%%%%%%
    \item[$(d\geq 1)$] Trivially, if $\IMP_0(\Gamma)$ is coNP-complete then $\IMP_d(\Gamma)$ is coNP-complete for every $d\geq 0$. It follows that in order to determine the complexity of $\IMP_d(\Gamma)$, for $d\geq 1$, the languages that still need to be discussed are those $\Gamma$ for which the corresponding problem $\CSP(\Gamma)$ is polynomial-time tractable. 
\end{itemize}

If the solution space of every relation in $\Gamma$ is closed under the ternary $\Minority$ operation, then in \cite{BM20} Bharathi and the author show that $\IMP_d(\Gamma)$ can be solved in $n^{O(d)}$ time for any $d\geq 1$.\footnote{In an earlier version of this paper appeared in SODA'19, I incorrectly declared that this case have been resolved by a previous result. I thank Andrei Bulatov, Akbar Rafiey, Standa \v{Z}ivn\'{y} and the anonymous referees for pointing out this issue.}

In this paper we prove the following theorem.
\begin{theorem}\label{th:result1}
 Let $\Gamma$ be a finite Boolean constraint language. If the solution space of every relation in $\Gamma$ is closed under any of $\{\Majority,\Max,\Min\}$ operations then $\IMP_d(\Gamma)$ can be solved in $n^{O(d)}$ time for any $d\geq 1$.
\end{theorem}
This yields an efficient algorithm for the membership problem  since the size of the input polynomial $f$ is $n^{O(d)}$. We remark that ``sparse'' polynomials are also discussed in this paper to some extent, however not in their full generality.
This permits us to avoid certain technicalities and discussion of how polynomials are represented. Moreover, these cases are not of prime interest for the $\sos$ applications that we have in mind where $d=O(1)$.

By Theorem~\ref{th:result1}, \cite[Theorem 1.1]{BM20} and Schaefer's Dichotomy Theorem~\ref{th:schaefer}, it follows that the only case that has yet to be investigated in order to understand the complexity of $\IMP_d(\Gamma)$ is when the solution space of every relation in $\Gamma$ is closed under the constant operation $c\in \{0,1\}$, but not simultaneously closed under any of $\Majority$, $\Minority$, $\Min$ or $\Max$ operations (otherwise $\IMP_d(\Gamma)$ is polynomial-time tractable). Under the latter case, note that $\CSP(\Gamma)$ is polynomial-time tractable either when the solution space of every relation in $\Gamma$ is closed under only \emph{one} constant operation or when it is closed, simultaneously, under \emph{both} constant operations $1$ and $0$.  These are the only cases left that are discussed in the following result that is proved within this paper.

\begin{theorem}\label{th:hardness}
  Let $\Gamma_1$ be a Boolean language with the solution space of every constraint closed under one constant operation $c \in\{0,1\}$.
  Let $\Gamma_2$ be a Boolean language with the solution space of every constraint closed under both constant operations $1$ and $0$. Assume that $\Gamma_1$ and $\Gamma_2$ are not simultaneously closed under any of $\Majority$, $\Minority$, $\Min$ or $\Max$ operations. Then, for $i\in\{1,2\}$, the problem $\IMP_i(\Gamma_i)$ is coNP-complete. If the constraint language $\Gamma_2$ has the operation $\neg$ as a polymorphism then the problem $\IMP_1(\Gamma_2)$ is coNP-complete.
\end{theorem}

The only case that Theorem~\ref{th:hardness} does not cover is the complexity of $\IMP_1(\Gamma_2)$ when $\Gamma_2$ is not simultaneously closed under any of $\{\Majority,\Minority,\Min,\Max,\neg\}$. Note that $\IMP_2(\Gamma_2)$ is coNP-complete but we do not know if $\IMP_1(\Gamma_2)$ is coNP-complete as well. However, we do not expect that $\IMP_1(\Gamma_2)$ is solvable in polynomial time (see the discussion in Section~\ref{sect:necessary}). We leave this problem as an open question.

Putting together the results from \cite{BM20} and \cite{Schaefer78} with Theorem~\ref{th:result1} and Theorem~\ref{th:hardness} 
we obtain the following corollary.

\begin{corollary}\label{th:summary}
 Let $\Gamma$ be a finite Boolean constraint language. If the solution space of every relation in $\Gamma$ is closed under any of $\{\Majority,\Minority,\Max,\Min\}$ operations, then $\IMP_d(\Gamma)$ can be solved in $n^{O(\max(d,1))}$ time for any $d\geq 0$.
 Otherwise there is a constant $d\in\{0,1,2\}$ such that $\IMP_d(\Gamma)$ is coNP-complete. 
\end{corollary}
 
Finally, the notion of \emph{pp-definability} is central in $\CSP$ theory.
We conclude the paper by discussing the correspondence between pp-definability and elimination ideals in  algebraic geometry.

%%%%%%%%%%%%%%%%%%%%%%%%%%%%%%%%%%%%%%%%%%%%%%%%%%%
\subsubsection{Some Applications}
%%%%%%%%%%%%%%%%%%%%%%%%%%%%%%%%%%%%%%%%%%%%%%%%%%%
In \cite{BM20} (for $\Minority$) and in the proof of Theorem~\ref{th:result1} it is shown that if the solution space of every relation in $\Gamma$ is closed under any of $\{\Majority,\Minority,\Max,\Min\}$ operations, then for any given $\emph{\CSP}(\Gamma)$ instance~$\Cc$
we can efficiently compute the bounded degree polynomials of the reduced {\GB} basis (assuming a grlex ordering) for the combinatorial ideal $\I_\Cc$.\footnote{Note that this is considerably different from the  bounded degree version of Buchberger's algorithm considered in~\cite{CleggEI96}.} This set of polynomials is 1-effective for $\I_\Cc$. Note that if the aforementioned conditions are not met, then the sufficient criteria by Raghavendra and Weitz~\cite{RaghavendraW17} cannot be efficiently applied for the instances of $\CSP(\Gamma)$. So we obtain an answer to Question~\ref{Q:effective} for constraint language problems.
 \begin{corollary}
   For Boolean constraint languages $\Gamma$, if the solution space of every relation in $\Gamma$ is closed under any of $\{\Majority,\Minority,\Max,\Min\}$ operations, then a 1-effective generating set for the combinatorial ideal $\I_\Cc$ can be computed in $n^{O(d)}$ time, for any given $\emph{\CSP}(\Gamma)$ instance~$\Cc$ and for all input polynomials of degree at most $d\geq 1$. Otherwise, computing $O(1)$-effective generating sets is coNP-complete for some $d\in\{0,1,2\}$.
 \end{corollary}
Moreover, our result implies necessary and sufficient conditions (assuming $P\not = NP$) for the efficient computation of Theta Body SDP relaxations, identifying therefore the borderline of tractability for constraint language problems.
This is summarized by the following corollary.

\begin{corollary}
For Boolean constraint languages $\Gamma$, if the solution space of every relation in $\Gamma$ is closed under any of $\{\Majority,\Minority,\Max,\Min\}$ operations, then  $\ThB_d(\I_\Cc)$ can be \emph{formulated} and \emph{solved} in $n^{O(d)}$ time (to high accuracy, with polynomial bit complexity) for any given $\emph{\CSP}(\Gamma)$ instance~$\Cc$. Otherwise, \emph{formulating} the $d$-th Theta Body SDP relaxation is coNP-complete for some $d\in\{0,1,2\}$.
\end{corollary}

 %a generating set $F=\{f_1,\ldots,f_r\}$ of the ideal $\I_\Cc$ corresponding to a given $\emph{\CSP}(\Gamma)$ instance $\Cc$ that is 1-effective for solving the membership problem of any polynomial of degree $\leq d$.
%assume that the polymorphism clone of $\Gamma$ is an idempotent clone that contains a non-projection. Then, for any fixed $d$, we can efficiently compute a generating set $F=\{f_1,\ldots,f_r\}$ of the ideal $\I_\Cc$ corresponding to a given $\emph{\CSP}(\Gamma)$ instance $\Cc$ that is 1-effective for solving the membership problem of any polynomial of degree $\leq d$. This and Theorem~1 in~\cite{RaghavendraW17} imply $\sos$ proofs with polynomial bit complexity for this class of problems. Otherwise, $\IMP_d(\Gamma)$ is NP-complete and the sufficient condition in~\cite{RaghavendraW17} cannot be efficiently applied.

\paragraph{Paper Structure:}
Throughout this paper we assume that the reader has some basic knowledge of both, $\CSP$ over a constraint language and algebraic geometry, more specifically \GB bases.
We use notation and basic properties as in standard textbooks and literature~\cite{Cox:2015,2017dfu7}.
However, in order to make this article as self-contained as possible and accessible to non-expert readers,
Section~\ref{sect:preliminaries} provides the essential context needed with the adopted notation. We recommend the non-expert reader to start with that section.
More precisely, Section~\ref{sect:csp_intro} gives a brief introduction to $\CSP$ over a constraint language with its algebra of polymorphisms. We refer to~\cite{Chen09,2017dfu7} for more details.
Section~\ref{sect:GBbasics} provides some rudiments of \GB bases and a coverage of the adopted notation. We refer to \cite{Cox:2015} for a more satisfactory introduction and for the missing details. The link between polynomial ideals and $\CSP$ is given in Section~\ref{sect:idealCSP}.

The main theorems of this paper, namely Theorem~\ref{th:result1} and Theorem~\ref{th:hardness}, give \emph{sufficient} and \emph{necessary} conditions to ensure the {Ideal Membership Problem} tractability. The sufficiency part is discussed in sections~\ref{sect:tract}, \ref{sect:majority} and \ref{sect:min_max}, with Section~\ref{sect:tract} giving an overview of the proof, main ideas and techniques. In particular in Section~\ref{sect:interlacing} it is provided a technical lemma that will be at the heart of the subsequent proofs. We believe that this lemma will be useful for generalizing this paper results to the finite domain case. The necessity part is considered in Section~\ref{sect:necessary}. The algebraic geometry point of view of $\CSP$s is further investigated in Section~\ref{sect:ppdef_zariski}, where the correspondence between pp-definability and elimination ideals is discussed.

\section{Background and Notation}\label{sect:preliminaries}

\subsection{Constraint Satisfaction and Polymorphisms}\label{sect:csp_intro}
%%%%%%%%%%%%%%%%%%%%%%%%%%%%%%%%%%%%%%%%%%%%%%%%%%%%%%%%%%
This section provides the reader with the essential context needed on $\CSP$s. For a more comprehensive introduction and the missing proofs we recommend~\cite{Chen09,2017dfu7} and the references therein.
\begin{definition}\label{def:constraint_language}
Let $D$ denote a finite set (\emph{domain}).
 By a $k$-ary \textbf{\emph{relation}} $R$ on a domain $D$ we mean a subset of the $k$-th cartesian power $D^k$; $k$ is said to be the \emph{arity} of the relation. A \textbf{\emph{constraint language}} $\Gamma$ over $D$ is a set of relations over $D$. A constraint language is \textbf{\emph{finite}} if it contains finitely many relations, and is \emph{Boolean} if it is over the two-element domain $\{0,1\}$.
\end{definition}
\begin{definition}\label{def:constr_over_language}
  A \emph{\textbf{constraint}} over a constraint language $\Gamma$ is an expression $R(x_1,\ldots, x_k)$ where $R$ is a relation of arity $k$ contained in $\Gamma$, and the $x_i$ are variables. A constraint is satisfied by a mapping $\phi$ defined on the $x_i$ if $(\phi(x_1),\ldots, \phi(x_k))\in R$.
\end{definition}

It is sometimes convenient to work with the corresponding \emph{predicate} which is a mapping from $D^k$ to $\{true,false\}$ specifying which tuples are in $R$: we will use both formalisms, so $(a,b,c)\in R$ and $R(a,b,c)$ both mean that the triple $(a,b,c)\in D^3$ is from the relation $R$. Analogously, a constraint $R(x_1,\ldots, x_k)$ is a subset of the cartesian product of the domains of the variables $x_1,\ldots, x_k$ such that each member  is in $R$.
\begin{definition}\label{def:csp}
  The \emph{(nonuniform) \textsc{Constraint Satisfaction Problem} ($\CSP$)} associated with language $\Gamma$ over $D$ is the problem $\emph{\CSP}(\Gamma)$ in which:
   %\begin{itemize}
   %\item
   an instance is a triple $\Cc=(X,D,C)$ where $X=\{x_1,\ldots,x_n\}$ is a set of $n$ variables and $C$  is a set of constraints over $\Gamma$ with variables from $X$.
   %\item
   The goal is to decide whether or not there exists a solution, i.e. a mapping $\phi: X\rightarrow D$ satisfying all of the constraints. We will use $Sol(\Cc)$ to denote the set of solutions of $\Cc$.
   %\end{itemize}
\end{definition}

%\paragraph{pp-definability:}

\begin{definition}\label{def:polymorph}
An operation $f:D^m \rightarrow D$ is a \textbf{\emph{polymorphism}} of a relation $R\subseteq D^k$ if for any choice of $m$ tuples from $R$, it holds that the tuple obtained from these $m$ tuples by applying $f$ coordinate-wise is in $R$.
If this is the case we also say that $f$
\emph{preserves} $R$, or that $R$ is \emph{invariant} or \emph{closed} with respect to $f$.
A polymorphism of a constraint language $\Gamma$ is an operation
that is a polymorphism of \emph{every} $R\in \Gamma$.
We use $\Pol(\Gamma)$ to denote the set of all polymorphisms of $\Gamma$.  
\end{definition}
This algebraic object $\Pol(\Gamma)$ has the following two properties.
\begin{itemize}
  \item $\Pol(\Gamma)$ contains all \emph{projections}, i.e. operations of the form $\pi_i(a_1,\ldots,a_n)=a_i$.
  \item $\Pol(\Gamma)$ is closed under composition.
\end{itemize}
Sets of operations with these properties are called {\emph{clones}}; therefore we refer to $\Pol(\Gamma)$ as the \emph{clone of polymorphisms} of $\Gamma$. 
\begin{lemma}(see e.g.~\cite{2017dfu7})
For constraint languages $\Gamma$, $\Delta$, where $\Gamma$ is finite, if every polymorphism of $\Delta$ is also a polymorphism of $\Gamma$, then $\emph{\CSP}(\Gamma)$ is polynomial time reducible to $\emph{\CSP}(\Delta)$.
\end{lemma}
\begin{definition}
  We say that an operation $f:D^k\rightarrow D$ is \emph{idempotent} if $f(d,\ldots,d)=d$ for all $d\in D$.
\end{definition}
%%%%%%
%%%%%
\begin{lemma}\cite{Post41}\label{th:post}
Every idempotent clone on $D = \{0, 1\}$ that contains a non-projection contains one of the following operations: the binary \Max, the binary \Min, the ternary \Majority, or the ternary \Minority.
\end{lemma}
In this paper we focus on Boolean $\CSP$s.
%
%%%%%%%%%%%%%%%%%%%%%%%%%%%%%%%%%%%%%%%%%%
%\subsection{Schaefer's dichotomy result}
%%%%%%%%%%%%%%%%%%%%%%%%%%%%%%%%%%%%%%%%%
In 1978 Schaefer~\cite{Schaefer78} obtained an interesting classification of the polynomial-time decidable cases of $\CSP$s when $D=\{0,1\}$.
Schaefer's dichotomy theorem was originally formulated in terms of properties of relations; here we give a modern presentation of the theorem that uses polymorphisms (see Jeavons~\cite{Jeavons98} and \cite{Chen09,2017dfu7}).
\begin{theorem}[Schaefer's Dichotomy Theorem~\cite{Schaefer78}]\label{th:schaefer}
  Let $\Gamma$ be a finite Boolean constraint language. Then
    the problem $\CSP(\Gamma)$ is polynomial-time tractable if its polymorphism clone contains a constant unary operation or
    it is an idempotent clone that contains a non-projection. Otherwise the problem is NP-complete.
%    one of the following operations: the \textbf{\emph{constant}} operation 0 or 1, the Boolean \textbf{\emph{AND}} operation, the Boolean  \emph{\textbf{OR}} operation, the ternary operation \textbf{\emph{\Majority}}, the ternary operation \emph{\textbf{minority}}.
%    Otherwise, the problem $\CSP(\Gamma)$ is NP-complete.
\end{theorem}
%

%
%   The problem $CSP(\Gamma)$ is polynomial-time tractable if $\Gamma$ has one of the following operations as polymorphisms:
%\begin{itemize}
%  \item the constant operation 0,
%  \item the constant operation 1,
%  \item the Boolean AND operation (weakly negative, {\sc{Horn-Sat}}),
%  \item the Boolean OR operation (weakly positive, {\sc{Dual Horn-Sat}}),
%  \item the ternary operation \Majority ({\sc 2-Sat}),
%  \item the ternary operation minority.
%\end{itemize}
%Otherwise, the problem $CSP(\Gamma)$ is NP-complete.
%\begin{theorem}~\cite{Schaefer78}
%  Let $\Gamma$ be a finite Boolean constraint language. The problem $CSP(\Gamma)$ is polynomial time tractable if at least one of the following conditions holds:
%  \begin{enumerate}[(a)]
%    \item Every relation in $\Gamma$ is satisfied when all variables are~0.
%    \item Every relation in $\Gamma$ is satisfied when all variables are 1.
%    \item Every relation in $\Gamma$ is definable by a CNF formula in which each conjunct has at most one negated variable.
%    \item Every relation in $\Gamma$ is definable by a CNF formula in which each conjunct has at most one unnegated variable.
%    \item Every relation in $\Gamma$ is definable by a CNF formula having at most 2 literals in each conjunct.
%    \item Every relation in $\Gamma$ is the set of a system of linear equation over the two elemnt field $\{0,1\}$.
%  \end{enumerate}
%  Otherwise the problem $CSP(\Gamma)$ is $NP$-complete.
%\end{theorem}

%%%%%%%%%%%%%%%%%%%%%%%
\subsection{Ideals, Varieties and Constraints}\label{sect:background}
Let $\Field$ denote an arbitrary field (for the applications of this paper $\Field=\Real$). Let $\Field[x_1, \ldots, x_n]$ be the ring of polynomials over a field $\Field$ and indeterminates $x_1,\ldots, x_n$. Let $\Field[x_1, \ldots, x_n]_d$ denote the subspace of polynomials of degree at most $d$.%The most commonly used fields will be: $\Real$, $\Complex$, $\Int$ and $\Rational$.
%%%%%%%%%%%%%%%%%%%%%%
%\paragraph{Ideals.}
%%%%%%%%%%%%%%%%%%%%%%
\begin{definition}\label{def:ideal}
The ideal (of $\Field[x_1,\ldots,x_n]$) generated by a finite set of polynomials $\{f_1,\ldots, f_m\}$ in $\Field[x_1,\ldots,x_n]$ is defined as
$$\Ideal{ f_1,\ldots,f_m}\mydef \left\{\sum_{i=1}^m t_i f_i\ \mid \ t_1,\ldots,t_m\in \Field[x_1,\ldots,x_n]\right\}.$$
The set of polynomials that vanish in a given set $S\subset \Field^n$ is called the \textbf{\emph{vanishing ideal}} of $S$ and denoted:
$\Ideal{S}\mydef \{f\in \Field[x_1,\ldots,x_n]\mid f(a_1,\ldots,a_n)=0 \  \forall (a_1,\ldots,a_n)\in S\}$.
\end{definition}
\begin{definition}
The \emph{radical} of an ideal $\I$ is an ideal such that an element $x$ is in the radical if and only if some power of $x$ is in $\I$. The radical of an ideal $\I$ is denoted by $\sqrt{\I}$.
An ideal $\I$ is \textbf{\emph{radical}} if $f^m \in\I$ for some integer $m\geq 1$ implies that $f\in \I$, namely if $\I=\sqrt{\I}$. 
\end{definition}
Another common way to denote $\Ideal{ f_1,\ldots,f_m}$ is by $\langle f_1,\ldots,f_m \rangle$ and we will use both notations interchangeably.
\begin{definition}
%\end{definition}
%
%\begin{definition}
 Let $\{f_1,\ldots, f_m\}$ be a finite set of polynomials in $\Field[x_1,\ldots,x_n]$. We call
$$\Variety{ f_1,\ldots,f_m}\mydef \{(a_1,\ldots,a_n)\in \Field^n\mid f_i(a_1,\ldots,a_n)=0, \  \text{with }1\leq i\leq m\}$$ 
the \textbf{\emph{affine variety}} defined by $f_1,\ldots, f_m$.
\end{definition}
\begin{definition}
  Let $\I\subseteq \Field[x_1,\ldots,x_n]$ be an ideal. We will denote by $\Variety{\I}$ the set $\Variety{\I}=\{(a_1,\ldots,a_n)\in \Field^n\mid f(a_1,\ldots,a_n)=0 \  \forall f\in \I\}$.
\end{definition}
\begin{theorem}[\cite{Cox:2015}, p. 196]\label{th:ideal_intersection}
  If $I$ and $J$ are ideals in $\Field[x_1,\ldots,x_n]$, then $\Variety{I\cap J}= \Variety{I}\cup \Variety{J}$.
\end{theorem}

\subsubsection{The Ideal-CSP correspondence}\label{sect:idealCSP} Constraints are in essence varieties, see e.g.~\cite{vandongenPhd,JeffersonJGD13}. %Following~\cite{vandongenPhd}, we shall translate CSPs to polynomial ideals and back.
Indeed, let $\Cc=(X,D,C)$ be an instance of the $\CSP(\Gamma)$ (see Definition~\ref{def:csp}). %Without loss of generality, we shall assume that $D\subset \N$ and $D\subseteq \Field$.
Let $Sol(\Cc)$ be the (possibly empty) set of all feasible solutions of $\Cc$.
In the following, we map $Sol(\Cc)$ to an ideal $\I_\Cc\subseteq \Field[X]$ such that $Sol(\Cc)=\Variety{\I_\Cc}$.

Let $Y=(x_{i_1},\ldots,x_{i_k})$ be a $k$-tuple of variables from $X$
and let $R(Y)$ be a non empty constraint from $C$.
In the following, we map $R(Y)$ to a generating system of an ideal such that the projection of the variety of this ideal onto $Y$ is equal to $R(Y)$ (see~\cite{vandongenPhd} for more details).

Every $v=(v_1,\ldots,v_k)\in R(Y)$ corresponds to some point $v\in \Field^k$. It is easy to check~\cite{Cox:2015} that $\Ideal{\{v\}}= \GIdeal{x_{i_1}-v_{1},\ldots,x_{i_k}-v_{k}}$, where $\GIdeal{x_{i_1}-v_{1},\ldots,x_{i_k}-v_{k}}$ is a radical ideal.
By Theorem~\ref{th:ideal_intersection}, we have
\begin{align}\label{eq:constr=var}
  R(Y)&=\bigcup_{v\in R(Y)} \Variety{\Ideal{\{v\}}}=\Variety{\I_{R(Y)}}, \\
  \I_{R(Y)} &= \bigcap_{v\in R(Y)} \Ideal{\{v\}},\nonumber
\end{align}
where $\I_{R(Y)}\subseteq \Field[Y]$ is zero-dimensional and radical ideal since it is the intersection of radical ideals (see~\cite[Proposition~16, p.197]{Cox:2015}). Equation~\eqref{eq:constr=var} states that constraint $R(Y)$ is a variety of $\Field^k$. It is easy to find a generating system for $\I_{R(Y)}$:
 \begin{align*}%\label{eq:genIConstr}
   \I_{R(Y)}&=\GIdeal{\prod_{v\in R}(1-\prod_{j=1}^{k}\delta_{v_j}(x_{i_j})),\prod_{j\in D}(x_{i_1}-j),\ldots,\prod_{j\in D}(x_{i_k}-j)},
 \end{align*}
where $\delta_{v_j}(x_{i_j})$ are indicator polynomials, i.e. equal to one when $x_{i_j}=v_j$ and zero when $x_{i_j}\in D\setminus\{v_j\}$; polynomials $\prod_{j\in D}(x_{i_k}-j)$ force variables to take values in $D$ and will be denoted as {\emph{domain polynomials}}.

The smallest ideal (with respect to inclusion) of $\Field[X]$ containing $\I_{R(Y)}\subseteq \Field[\x]$ will be denoted $\I_{R(Y)}^{\Field[X]}$ and it is called the $\Field[X]$-module of $\I$. The set $Sol(\Cc)\subset \Field^n$ of solutions of $\Cc=(X,D,C)$ is the intersection of the varieties of the constraints:
\begin{align}
  Sol(\Cc) & = \bigcap_{R(Y)\in C}\Variety{\I_{R(Y)}^{\Field[X]}}=\Variety{\I_C},\nonumber\\ 
  \I_\Cc&=\sum_{R(Y)\in C}\I_{R(Y)}^{\Field[X]}.\label{eq:IC}
\end{align}

The following properties follow from Hilbert's Nullstellensatz. %A simple direct proof is given in Section~\ref{sect:hilbert01}.
\begin{theorem}\label{th:nullstz}
Let $\Cc$ be an instance of the $\CSP(\Gamma)$ and $\I_\Cc$ defined as in~\eqref{eq:IC}. Then
  \begin{align}
    &\Variety{\I_\Cc}=\emptyset \Leftrightarrow 1\in \Ideal{\I_\Cc} \Leftrightarrow \I_\Cc=\Field[X],  &\text{(Weak Nullstellensatz)}\label{eq:weak_nstz}\\
    &\Ideal{\Variety{\I_\Cc}}=\sqrt{\I_\Cc}, &\text{(Strong Nullstellensatz)}\label{eq:strong_nstz}\\
    &\sqrt{\I_\Cc}=\I_\Cc. &\text{(Radical Ideal)}\label{eq:ICradical}
  \end{align}
\end{theorem}

Theorem~\ref{th:nullstz} follows from a simple application of the celebrated and basic result in algebraic geometry known as Hilbert's Nullstellensatz. In the general version of Nullstellensatz it is necessary to work in an algebraically closed field and take a radical of the ideal of polynomials. In our special case it is not needed due to the presence of domain polynomials. Indeed, the latter implies that we know a priori that the solutions must be in $\Field$ (note that we are assuming $D\subseteq \Field$).

\subsection{\GB bases}\label{sect:GBbasics}
%%%%%%%%%%%%%%%%%%%%%%%%%%%%%%%%%%%%
% The most commonly used fields will be: $\Real$, $\Complex$, $\Int$ and $\Rational$.
%We will use the following terminology (see \cite{Cox:2015}).
%
We can reconstruct the monomial $x^\alpha=x_1^{\alpha_1}\cdots x_n^{\alpha_n}$ from the $n$-tuple of exponents $\alpha =(\alpha_1,\ldots,\alpha_n)\in \Zz^n_{\geq0}$. This establishes a one-to-one correspondence between the monomials in $\Field[x_1,\ldots,x_n]$ and $\Zz^n_{\geq0}$. Any ordering $>$ we establish on the space $\Zz^n_{\geq0}$
will give us an ordering on monomials: if $\alpha > \beta$ according to this ordering, we will also say that $x^\alpha > x^\beta$.
For more details we refer to \cite[Definition 1, p.55]{Cox:2015}.

In this paper we will be mainly interested in two monomial orderings, namely the \emph{lexicographic order} (\textbf{lex}) and the \emph{graded lexicographic order} (\textbf{grlex}): 
\begin{itemize}
    \item In the lex ordering, $x^\alpha >_{\mathsf{lex}} x^\beta$ if the left-most nonzero entry of $\alpha-\beta$ is positive. Notice that a particular order of the variables is assumed.
    \item In the grlex ordering, $x^\alpha >_{\mathsf{grlex}} x^\beta$ if $|\alpha|>|\beta|$ or if $|\alpha|=|\beta|$ and $x^\alpha >_{\mathsf{lex}} x^\beta$, where $|\gamma|=\sum_i \gamma_i$.
\end{itemize}

In the remainder of this section we suppose a fixed monomial ordering $>$ on $\Field[x_1,\ldots,x_n]$, which will not be defined explicitly.

\begin{definition}
  For any $\alpha=(\alpha_1,\cdots,\alpha_n)\in \Zz^n_{\geq0}$ let $x^\alpha\mydef \sum_{i=1}^{n}x_i^{\alpha_i}$. Let $f= \sum_{\alpha} a_{\alpha}x^\alpha$ be a nonzero polynomial in $\Field[x_1,\ldots,x_n]$ and let $>$ be a monomial order.
  \begin{enumerate}
    \item The \textbf{\emph{multideg}} of $f$ is $\multideg(f)\mydef \max(\alpha\in \Zz^n_{\geq0}:a_\alpha\not = 0)$.
    \item The \textbf{\emph{leading coefficient}} of $f$ is $\LC(f)\mydef a_{\multideg(f)}\in \Field$.
    \item The \textbf{\emph{leading monomial}} of $f$ is $\LM(f)\mydef x^{\multideg(f)} \ \text{(with coefficient 1)}$.
    \item The \textbf{\emph{leading term}} of $f$ is $\LT(f)\mydef \LC(f)\cdot \LM(f)$.
  \end{enumerate}
\end{definition}
%In order to compute the \GB basis we will use the following known results and definitions. We refer to standard books (see e.g. \cite{becker1993grobner,Cox:2007}) for full details.
%

The concept of \emph{reduction}, also called \emph{multivariate division} or \emph{normal form computation}, is central to \GB basis theory. It is a multivariate generalization of the Euclidean division of univariate polynomials.

\begin{definition}\label{def:reduction}
Fix a monomial order and let $G=\{g_1,\ldots,g_t\}\subset \Field[x_1,\ldots,x_n]$. Given $f\in \Field[x_1,\ldots,x_n]$, we say that \emph{\textbf{$f$ reduces to $r$ modulo $G$}}, written
$f\rightarrow_G r$,
if $f$ can be written in the form
$f=A_1g_1+\dots+A_t g_t+r$ for some $A_1,\ldots,A_t,r\in \Field[x_1,\ldots,x_n]$,
such that:
\begin{enumerate}%[(i)]
  \item No term of $r$ is divisible by any of $\LT(g_1),\ldots,\LT(g_t)$.
  \item Whenever $A_i g_i\not=0$, we have $\multideg(f)\geq \multideg(A_ig_i)$.
\end{enumerate}
The polynomial remainder $r$ is called a \emph{\textbf{normal form of $f$ by $G$}} and will be denoted by $f|_G$.
\end{definition}

A normal form of $f$ by $G$, i.e. $f|_G$, can be obtained by repeatedly performing the following until it cannot be further applied: choose any $g\in G$ such that $\LT(g)$ divides some term $t$ of $f$ and replace $f$ with $f-\frac{t}{\LT(g)}g$. Note that the order in which we choose the polynomials $g$ in the division process is not specified.

In general a normal form $f|_G$ is not uniquely defined.
Even when $f$ belongs to the ideal generated by $G$, i.e. $f\in \Ideal{G}$, it is not always true that $f|_G=0$.
\begin{example}
  Let $f=xy^2-y^3$ and $G=\{g_1,g_2\}$, where $g_1=xy-1$ and $g_2=y^2-1$. Consider the graded lexicographic order (with $x>y$) and note that
 % \begin{align*}
    $f = y\cdot g_1 - y\cdot g_2 + 0$ and
    $f = 0\cdot g_1 + (x-y)\cdot g_2 + x-y$.
%  \end{align*}
 % From the first equation we see that the corresponding normal form of $f$ by $G$ is zero and therefore $f\in \Ideal{G}$. The normal form of $f$ by $G$ in the second equation is not zero.
\end{example}
This non-uniqueness is the starting point of \GB basis theory.
\begin{definition}
Fix a monomial order on the polynomial ring $\Field[x_1,\ldots,x_n]$. A finite subset $G = \{g_1,\ldots, g_t\}$ of an ideal $\I \subseteq \Field[x_1,\ldots,x_n]$ different from $\{0\}$ is said to be a \emph{\textbf{\GB basis}} (or \emph{\textbf{standard basis}}) if
$\langle \LT(g_1),\ldots, \LT(g_t)\rangle = \langle \LT(\I)\rangle$, where we denote by $\langle \LT(\I)\rangle$ the ideal generated by the elements of the set $\LT(\I)$ of leading terms of nonzero elements of $\I$.
\end{definition}
\begin{definition}\label{def:redGB}
A \textbf{\emph{reduced \GB basis}} for a polynomial ideal $\I$ is a \GB basis $G$ for $\I$ such that:
\begin{enumerate}%[(i)]
  \item $\LC(g)= 1$ for all $g \in G$.
  \item For all $g \in G$, no monomial of $g$ lies in $\GIdeal{\LT(G\setminus \{g\})}$.
\end{enumerate}
\end{definition}
It is known (see~\cite[Theorem~5, p. 93]{Cox:2015}) that for a given monomial ordering, a polynomial ideal $\I\not=\{0\}$ has a reduced \GB basis (see Definition~\ref{def:redGB}), and the reduced \GB basis is unique.
\begin{proposition}[\cite{Cox:2015}, Proposition~1, p. 83]\label{th:gbprop}
Let $\I\subset \Field[x_1,\dots,x_n]$ be an ideal and let $G=\{g_1,\ldots,g_t\}$ be a \GB basis for $\I$. Then given $f\in \Field[x_1,\dots,x_n]$, $f$ can be written in the form
$f=A_1g_1+\dots+A_t g_t+r$ for some $A_1,\ldots,A_t,r\in \Field[x_1,\ldots,x_n]$,
such that:
\begin{enumerate}%[(i)]
  \item No term of $r$ is divisible by any of $\LT(g_1),\ldots,\LT(g_t)$.
  \item Whenever $A_i g_i\not=0$, we have $\multideg(f)\geq \multideg(A_ig_i)$.
  \item There is a unique $r\in \Field[x_1,\dots,x_n]$.
\end{enumerate}
In particular, $r$ is the remainder on division of $f$ by $G$ no matter how the elements of $G$ are listed when using the division algorithm.%\footnote{For the description of the division algorithm we refer to~\cite{Cox:2015}, p.61.}.
\end{proposition}
\begin{corollary}[\cite{Cox:2015}, Corollary~2, p.84]\label{th:imp}
  Let $G=\{g_1,\ldots,g_t\}$ be a \GB basis for $\I\subseteq \Field[x_1,\dots,x_n]$ and let $f\in \Field[x_1,\dots,x_n]$. Then $f\in \I$ if and only if the remainder on division of $f$ by $G$ is zero.
\end{corollary}
\begin{definition}\label{def:pdiv}
We will write $\reduce f F$ for the remainder of $f$ by the ordered $s$-tuple $F=(f_1,\ldots,f_s)$. If $F$ is a $\GB$ basis for $\spn{f_1,\dots,f_s}$, then we can regard $F$ as a set (without any particular order) by Proposition~\ref{th:gbprop}.
\end{definition}
%

%\begin{definition}
%A \textbf{remainder} of a polynomial $f\in \Field[x_1,\ldots,k_n]$ on division by set $F=\{f_1,\ldots,f_s\}$ of polynomials from $\Field[x_1,\ldots,k_n]$, denoted by $f|_F$, is computed by repeatedly performing the following until it cannot be further applied: Choose any $i \in \{1, . . . ,s\}$ such that $\LT(f_i)$ divides some term $\tau$ of $f$ and set $f := f - \frac{\tau}{\LT(f_i)}f_i$.
%\end{definition}
%
The ``obstruction'' to $\{g_1,\ldots, g_t\}$ being a \GB basis is the possible occurrence of polynomial combinations of the $g_i$ whose leading terms are not in the ideal generated by the $\LT( g_i)$. One way (actually the only way) this can occur is if the leading terms in a suitable combination cancel, leaving only smaller terms. The latter is fully captured by the so called $S$-polynomials that play a fundamental role in \GB basis theory.
\begin{definition}\label{def:spoly}
  Let $f,g\in \Field[x_1,\ldots,x_n]$ be nonzero polynomials.
  %\begin{enumerate}[(i)]
   % \item
    If $\multideg(f) =\alpha$ and $\multideg(g)= \beta$, then let $\gamma=(\gamma_1,\ldots,\gamma_n)$, where $\gamma_i = \max(\alpha_i,\beta_i)$ for each $i$. We call $x^\gamma$ the \emph{\textbf{least common multiple}} of $\LM(f)$ and $\LM(g)$, written $x^\gamma = \LCM(\LM(f),\LM(g))$.
    %\item
    The \emph{\textbf{$S$-polynomial}} of $f$ and $g$ is the combination $S(f,g) = \frac{x^\gamma}{\LT(f)}\cdot f - \frac{x^\gamma}{\LT(g)}\cdot g$.
    %\begin{align*}
%      S(f,g) &= \frac{x^\gamma}{\LT(f)}\cdot f - \frac{x^\gamma}{\LT(g)}\cdot g
%    \end{align*}
  %\end{enumerate}
\end{definition}
The use of $S$-polynomials to eliminate leading terms of multivariate polynomials generalizes the row reduction algorithm for systems of linear equations. If we take a system of homogeneous linear equations (i.e.: the constant coefficient equals zero), then it is not hard to see that bringing the system in triangular form yields a \GB basis for the system.
%%%%%%%%%%%%%%%%%%%%%%%%%%%%%%%%%%%%%%%%%%%%%%%%%%%%
\begin{theorem}[\cite{Cox:2015}, Theorem 3, p.105, \textbf{Buchberger's Criterion}]\label{th:crit}
A basis $G=\{g_1,\ldots,g_t\}$ for an ideal $\I$ is a \GB basis if and only if $S(g_i,g_j)\rightarrow_G 0$ for all $i\not=j$.
\end{theorem}
%%%%%%%%%%%%%%%%%%%%%%%%%%%%%%%%%%%%%%%%%%%%%%%%%%%%%%%%%%%%%%%%
%
By Theorem~\ref{th:crit} it is easy to show whether a given basis is a \GB basis. Indeed, if $G$ is a \GB basis then given $f\in \Field[x_1,\dots,x_n]$, $f|_G$ is unique and it is the remainder on division of $f$ by $G$, no matter how the elements of $G$ are listed when using the division algorithm.

Furthermore, Theorem~\ref{th:crit} leads naturally to an algorithm for computing \GB bases for a given ideal $\I=\langle f_1,\ldots,f_s \rangle$: start with a basis $G=\{f_1,\ldots,f_s\}$ and for any pair $f,g\in G$ with $S(f,g)|_G\not= 0$ add $S(f,g)|_G$ to $G$.
This is known as Buchberger's algorithm~\cite{BuchbergerThesis} (for more details see Algorithm~\ref{Algo:buchbergerAlgo} in Section~\ref{sect:buchbergerAlgo}).

Note that Algorithm~\ref{Algo:buchbergerAlgo} is non-deterministic and the resulting \GB basis in not uniquely determined by the input. This is because the normal form $S(f,g)|_G$ (see Algorithm~\ref{Algo:buchbergerAlgo}, line~\ref{eq:reminder}) is not unique as already remarked.
We observe that one simple way to obtain a deterministic algorithm (see \cite{Cox:2015}, Theorem~2, p. 91) is to replace $h:=S(f,g)|_G$ in line~\ref{eq:reminder} with $h:=\reduce {S(f,g)} G$ (see Definition~\ref{def:pdiv}), where in the latter $G$ is an ordered tuple. However, this is potentially dangerous and inefficient. Indeed, there are simple cases where the combinatorial growth of set $G$ in Algorithm~\ref{Algo:buchbergerAlgo} is out of control very soon.

%%%%%%%%%%%%%%%%%%%%%%%%%%%%%%%%%%%%
%%%%%%%%%%%%%%%%%%%%%%%%%%%%%%%%%%%%%%%%%%%%%%%%%%%%%%%%%%%%%%%%%%%%%%%%%%%%%%%

%%%%%%%%%%%%%%%%%%%%%%%%%%%%%%%%%%%%%%%%%%%%%%%%%%%%%%%%%%%
%\section{}\label{sect:appendixA}
%%%%%%%%%%%%%%%%%%%%%%%%%%%%%%%%%%%%%%%%%%%%%%%%%%%%%%%%%%%

%%%%%%%%%%%%%%%%%%%%%%%%%%%%%%%%%%%
\subsubsection{\GB basis construction}\label{sect:buchbergerAlgo}
%%%%%%%%%%%%%%%%%%%%%%%%%%%%%%%%%%%%%
Buchberger's algorithm~\cite{BuchbergerThesis} can be formulated as in Algorithm~\ref{Algo:buchbergerAlgo}. %%%%%%%%%%%%%%%%%%%%%%%%%%%%%%%%
\begin{algorithm}
\caption{Buchberger's Algorithm}
\label{Algo:buchbergerAlgo}
\begin{algorithmic}[1]
\STATE \textbf{Input}: A finite set $F=\{f_1,\ldots,f_s\}$ of polynomials
\STATE \textbf{Output}: A finite \GB basis $G$ for $\spn{f_1,\ldots,f_s}$
\STATE $G:= F$\label{eq:startingG}
\STATE $C:= G\times G$
\WHILE{$C\not = \emptyset$}
\STATE Choose a pair $(f,g)\in C$
\STATE $C:=C\setminus \{(f,g)\}$
\STATE $h:=S(f,g)|_G$ \label{eq:reminder}
\IF {$h\not=0$}\label{eq:Buch_cond}
\STATE $C:= C\cup (G\times \{h\})$
\STATE $G:= G\cup \{h\}$\label{eq:addspoly}
\ENDIF
\ENDWHILE
\STATE Return G
\end{algorithmic}
\end{algorithm}
%%%%%%%%%%%%%%%%%%%%%%%%%%%%%%%%
The pairs that get placed in the set $C$ are often referred to as \emph{critical pairs}. Every newly added reduced $S$-polynomial $h$ enlarges the set $C$. If we use $h:=\reduce {S(f,g)} G$ in line~\ref{eq:reminder} then there are simple cases where the situation is out of control. This combinatorial growth can be controlled to some extent be eliminating unnecessary critical pairs.

%%%%%%%%%%%%%%%%%%%%%%%%%%%%%%%%%%%%%%%%%%%%%%%%%%%%%%%%%%%
%%%%%%%%%%%%%%%%%%%%%%%%%%%%%%%%%%%%%%%%%%%%%%%%%%%%%%%%%%%
%%%%%%%%%%%%%%%%%%%%%%%%%%%%%%%%%%%%%%%%%%%%%%%%%%%%%%%%%%%

%%%%%%%%%%%%%%%%%%%%%%%%%%%%%%%%%%%%%%%%%%%%%%%%%%%%%%%%%%%%%%%%%%%%%%%%

%%%%%%%%%%%%%%%%%%%%%%%%%%%%%%%%%%%%%%%%%%%%%%

%\end{appendices}

\section{The Ideal Membership Problem: Tractability}\label{sect:tract}
%%%%%%%%%%%%%%%%%%%%%%%%%%%%%%%%%%%%%%%%%%%%%%%%%%%%%%%%%%%%%%%%%%%%%%%%%%%%%%%%%
%This paper focuses on polynomial ideals corresponding to \emph{Boolean} $\CSP(\Gamma)$ (see Section~\ref{sect:idealCSP}).
%
%Theorem~\ref{th:result1} gives \emph{sufficient} and \emph{necessary} conditions to ensure the \emph{Ideal Membership Problem} tractability. The sufficiency part is discussed in this section and in sections \ref{sect:majority} and \ref{sect:min_max}. The necessity part is considered in Section~\ref{sect:necessary}.
%We provide an overview of the proof with the main ideas and a technical lemma. These ideas are formally developed and used in Section~\ref{sect:majority} and Section~\ref{sect:min_max} where the proof of Lemma~\ref{th:sufficient} is completed.
%These sufficient restrictions are shown to be \emph{necessary} in Section~\ref{sect:necessary}, therefore completing the proof of Theorem~\ref{th:result1}.

We provide an overview of the proof of Theorem~\ref{th:result1} with the main ideas and a technical lemma. These ideas are further developed and used in Section~\ref{sect:majority} and Section~\ref{sect:min_max}.
We solve the membership question by using \GB bases techniques. 
A $\GB$ basis provides a representation of an ideal that allows us to easily decide membership (see Section~\ref{sect:GBbasics} and Corollary~\ref{th:imp}). 
%\GB bases can be computed via Buchberger's algorithm (see Section~\ref{sect:buchbergerAlgo}). An important question regarding Buchberger's algorithm is its complexity. We discuss this for combinatorial ideals corresponding to Boolean constraint languages.
\begin{remark}
  From now on, even where not explicitly written, we will assume that monomials are ordered according to the \emph{graded lexicographic order}, grlex for short, (see Section~\ref{sect:GBbasics} or \cite[Definition~5 on p. 58]{Cox:2015} for additional details). Other total degree orderings (i.e. ordered according to the total degree first) could be used with the same results.
\end{remark}

%%%%%%%%%%%%%%%%%%%%%%%%%%%%%%%%%%%%%%%%%%%%%%%%%%%%%%%
\subsection{Overview of the Proof of Theorem \ref{th:result1}}
%%%%%%%%%%%%%%%%%%%%%%%%%%%%%%%%%%%%%%%%%%%%%%%%%%%%%%%%
 We will make use of the following definition.
\begin{definition}\label{def:2terms}
  For a given set $X=\{x_1,\ldots,x_n\}$ of variables and for any set $S\subseteq [n]$ possibly empty, $\alpha\in \{0, \pm 1\}$, let a \textbf{\emph{term}} be defined as \footnote{The empty product has the value 1.}
  \begin{align*}
    \tau^+(S)&\mydef \alpha\prod_{i\in S}x_i, \quad \textsc{*positive term*} \\
    \tau^-(S)&\mydef\alpha\prod_{i\in S}(x_i-1). \quad \textsc{*negative term*}
  \end{align*}
  For $S_1,S_2\subseteq [n]$ and $i\in [n]$, let a \textbf{\emph{2-terms polynomial}} be a polynomial that is the sum of two terms or it is $\pm(x_i^2-x_i)$.
  We say that a set $G$ of polynomials is \textbf{\emph{2-terms structured}} if each polynomial from $G$ is a 2-terms polynomial.

  We further distinguish between the following special 2-terms polynomials:
  \begin{align*}
    \T^+ & \mydef \{\tau^+(S_1)+\tau^+(S_2)\mid S_1,S_2\subseteq[n]\} \cup\{\pm(x_i^2-x_i)\mid i\in [n]\},  \quad \textsc{*positive 2-terms*}\\
    \T^- & \mydef \{\tau^-(S_1)+\tau^-(S_2)\mid S_1,S_2\subseteq[n]\} \cup\{\pm(x_i^2-x_i)\mid i\in [n]\}. \quad \textsc{*negative 2-terms*}
  \end{align*}
  \end{definition}
  
\subsubsection{Main Ideas}
Let $G =\{g_1,\ldots, g_t\}$ be the \emph{reduced} \GB basis (see Definition~\ref{def:redGB}) for the combinatorial ideal $\I_\Cc$ corresponding to a given $\emph{\CSP}(\Gamma)$ instance $\Cc$. Recall (see~\cite[Theorem~5, p.93]{Cox:2015}) that for a given monomial ordering $G$ is unique. We assume that $\Pol(\Gamma)$ contains at least one of the three operations $\{\Max, \Min, \Majority\}$.
The proof of Theorem~\ref{th:result1} will show the following facts:
  \begin{itemize}
    \item The reduced \GB basis $G$ of $\I_C$ has the 2-terms structure (for grlex order);
    \item If $\Majority\in\Pol(\Gamma)$ then every 2-terms polynomial $g\in G$ has degree at most 2.
    \item If $\Max\in\Pol(\Gamma)$ then every $g\in G$ is a negative 2-terms polynomial (of arbitrarily large degree).
    \item If $\Min\in\Pol(\Gamma)$ then every $g\in G$ is a positive 2-terms polynomial (of arbitrarily large degree).
  \end{itemize}

If $\Majority\in\Pol(\Gamma)$ then the 2-terms characterization of the reduced \GB bases will be sufficient to guarantee the tractability of \GB basis computation.
A key part of the \GB basis algorithm is the computation of the so called $S$-polynomials in normal form (see definitions~\ref{def:spoly} and \ref{def:reduction} and Theorem~\ref{th:crit}).
We show how to compute $S(f,g)|_G$ (see Algorithm~\ref{Algo:buchbergerAlgo}, line~\ref{eq:reminder}) in such a way Buchberger's algorithm will take $n^{O(1)}$ time to compute a \GB basis.

If $\Min\in\Pol(\Gamma)$ or $\Max\in\Pol(\Gamma)$ then we will observe that
for any given 2-terms polynomial $p$ we can efficiently check whether $p\in \I_C$. This implies that we can compute the ``truncated'' reduced \GB basis $G_d=G\cap \Field[x_1,\ldots,x_n]_d$ in $n^{O(d+1)}$ time, for any degree $d$. If we ever wish to test membership in $\I_\Cc$ for some
polynomial $f$ of degree $d$, we need only to compute $G_d$ (assuming grlex order). Indeed, by Proposition~\ref{th:gbprop} and Corollary~\ref{th:imp}, the membership test can be computed by using only polynomials from $G_d$ and therefore we have
\begin{align*}%\label{eq:imp_GB}
  f\in \I_\Cc\cap \Field[x_1,\ldots,x_n]_d&\Leftrightarrow \reduce f {G_d}=0.
\end{align*}
This yields an efficient algorithm for the membership problem (the size of the input polynomial $f$ is $n^{O(d)}$).
In general the exponential dependence on the input polynomial degree $d$ is unavoidable since the input size can be $n^{\Omega(d)}$. However, we complement this result by showing that when $f$ is a ``sparse'' polynomial of high degree then we can remove the exponential dependance on $d$ by (i) either efficiently compute a subset (that depends on $f$) $G_f\subseteq \I_\Cc$ such that $\reduce f {G_f}=0$, (ii) or show a certificate that $f\not\in \I_\Cc$.

\paragraph{Techniques.} As discussed in Section~\ref{sect:GBbasics}, Theorem~\ref{th:crit} leads naturally to an algorithm, known as Buchberger's algorithm. A fundamental role is played by the $S$-polynomials in normal form, i.e. $S(f,g)|_G$ which is the building block to compute a \GB basis: $S(f,g)|_G$ is an operation that combines any two elements from the ideal to form a third polynomial from the ideal.
%
%Another way to characterize a \GB basis is by a closure property.
For a given $\I=\spn{f_1,\dots,f_s}$ we have that $G=\{g_1,\ldots, g_k\}$ is a \GB for $\I$ if $\spn{f_1,\dots,f_s}=\spn{g_1,\dots,g_k}$ and $S(f,g)|_G=0$ for every $f,g\in G$.
%set $G$ is closed under $S(f,g)|_G$ composition.

Recall that Buchberger's algorithm (see Section~\ref{sect:GBbasics}, Theorem~\ref{th:crit} and Algorithm~\ref{Algo:buchbergerAlgo}) is \emph{non-deterministic} because a normal form $S(f,g)|_G$ (see Definition~\ref{def:reduction}) is \emph{not unique} (unless $G$ is a {\GB} basis). Indeed, a normal form $S(f,g)|_G$ can be obtained by repeatedly performing the following until it cannot be further applied: choose any $g\in G$ such that $\LT(g)$ divides some term $t$ of $S(f,g)$ and replace $S(f,g)$ with $S(f,g)-\frac{t}{\LT(g)}g$. Note that the order in which we choose the polynomials $g$ in the division process is not specified.

The order in which we choose polynomials will play a fundamental role in this paper.
With this in mind, in the next section we present a technical lemma (the \emph{Interlacing Lemma}~\ref{th:interlacing-lemma}) that will be used to compute a ``special'' normal form $S(f,g)^*|_G$ that preserves the 2-terms structure. This implies that any reduced \GB basis is 2-terms structured. Indeed, for grlex monomial ordering, we compute a \GB basis $G$ for $\I_\Cc$ by using Buchberger's Algorithm~\ref{Algo:buchbergerAlgo} with the following change: at line~\ref{eq:reminder} of Algorithm~\ref{Algo:buchbergerAlgo}, replace $S(f,g)$ with $S(f,g)^*$ (see~\eqref{eq:intnormal}) and use the reduced by $G$ polynomial $S(f,g)^*|_G$.
Note that $S(f,g)^*|_{G}$ is a normal form of $S(f,g)$ by $G$, namely there is an ordering of the polynomials division that make $S(f,g)^*|_{G}=S(f,g)|_G$.
 Therefore Algorithm~\ref{Algo:buchbergerAlgo} with the above specified changes returns a \GB basis, since Buchberger's Algorithm is guaranteed to return a \GB basis independently on the order in which we perform polynomial divisions at line~\ref{eq:reminder}.

It follows that if the starting generators of the combinatorial ideal $\I_\Cc$ are 2-terms polynomials then the used $S(f,g)^*|_G$ operations will preserves this structure. The 2-terms structure of the reduced \GB bases follows by observing that division of 2-terms polynomials preserves the 2-terms structure property as well.

\begin{remark}
  Note that there are normal forms $S(f,g)|_G$ that do not guarantee the 2-terms structure.
\end{remark}

\subsection{The Interlacing Lemma}\label{sect:interlacing}
%%%%%%%%%%%%%%%%%%%%%%%%%%%%%%%%%%%%
The $S$-polynomials are the building blocks to compute a \GB basis (see Definition~\ref{def:spoly}). An $S$-polynomial combines any two elements from the ideal to form a third polynomial from the ideal. Every \GB basis can be computed by adding non-zero $S$-polynomials in normal form.

In the following we prove a key structural property (\emph{Interlacing Property}) of the $S$-polynomials in normal form that will be used several times and will be at the heart of the subsequent proofs. The Interlacing Property shows how two polynomials interlace in the corresponding $S$-polynomial.

 We believe that the Interlacing Property will be useful for generalizing this paper results to the finite domain case, as confirmed by preliminary investigations by the author.

\begin{figure}[hbtp]
    \centering
    \includegraphics[scale=1]{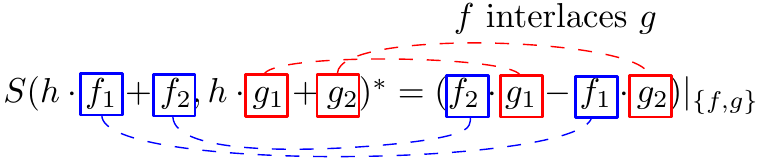}
    \caption{Interlacing property of $S$-polynomials ($\LC(h)=\LC(f_1)=\LC(g_1)=1$).}
  \label{fig:interlacing}
\end{figure}

%%%%%%%%%%%%%%%%%%%%%%%%%%%%%%%%%%%%%%%%%%%%%%%%%%%%
\begin{lemma}[Interlacing Lemma]\label{th:interlacing-lemma}
%Let $f_1,f_2,g_1,g_2,h\in \Field[x_1,\dots,x_n]$.
Let $>$ be a monomial order.
Suppose that we have $f,g\in \Field[x_1,\ldots,x_n]$ with $f= h\cdot f_1+f_2$ and $g= h\cdot g_1+g_2$ such that
\begin{align}
&h \cdot f_1\not= 0\label{int0f},\\
&h \cdot g_1\not= 0,\\
&\text{ if } f_2\not=0 \text{ then }\LM(h \cdot f_1)>\LM(f_2), \label{eq:int1}\\
&\text{ if } g_2\not=0 \text{ then }\LM(h \cdot g_1)>\LM(g_2), \label{eq:int2}\\
&\LCM(\LM(f_1),\LM(g_1))=\LM(f_1)\cdot \LM(g_1). \label{eq:int3}
\end{align}
Let
\begin{align}\label{eq:intnormal}
  S(f,g)^* &\mydef \frac{\left(f_2 \cdot g_1-f_1 \cdot g_2\right)|_{\{f,g\}}}{\LC(h)\cdot \LC(f_1)\cdot \LC(g_1)}.
\end{align}
Then $S(f,g)\rightarrow_{\{f,g\}} S(f,g)^*$ (Interlacing Property).
\end{lemma}
\begin{proof}
Note that $\LCM(\LM(f),\LM(g))=\LM(h)\cdot\LM(f_1)\cdot \LM(g_1)$.
By Definition~\ref{def:spoly} we have
\begin{align}
  S(f, g) & = \frac{\LM(g_1)}{\LC(f)}\cdot f - \frac{\LM(f_1)}{\LC(g)}\cdot g \nonumber\\
  & = \frac{g_1-(g_1-\LC(g_1)\cdot \LM(g_1))}{\LC(h)\cdot \LC(f_1)\cdot \LC(g_1)}\cdot f - \frac{f_1-(f_1-\LC(f_1)\cdot \LM(f_1))}{\LC(h)\cdot \LC(f_1)\cdot \LC(g_1)}\cdot g \nonumber\\
  & = {C}
  \left(f_2\cdot g_1 - f_1\cdot g_2 +(f_1-\LT(f_1))\cdot g-(g_1-\LT(g_1))\cdot f \right).\label{s_temp}
\end{align}
where $C=1/(\LC(h)\cdot \LC(f_1)\cdot \LC(g_1))$.
%%%%%%%%%%%%%%%%%%%%%%%

Let $q=f_2 \cdot g_1-f_1 \cdot g_2$.
For a normal form $q|_{\{f,g\}}$ of $q$ modulo $\{f,g\}$ the following holds (see Definition~\ref{def:reduction}):
\begin{enumerate}%[I]
  \item[(i)] $q=A_f f+A_g g+ q|_{\{f,g\}}$ for some $A_f,A_g\in \Field[x_1,\ldots,x_n]$.\label{cond1}
  \item[(ii)] No term of $q|_{\{f,g\}}$ is divisible by any of $\LT(f),\LT(g)$.\label{cond2}
  \item[(iii)] For any $p\in \{f,g\}$, whenever $A_p p\not=0$, we have $\multideg(q)\geq \multideg(A_p p)$.\label{cond3}
\end{enumerate}
%%%%%%%%%
Note that if $A_f f\not=0$ then (iii) implies that
\begin{align}\label{c3}
  \LM(g_1)>\LM(A_f).
\end{align}
Indeed, by contradiction assume $\LM(g_1)\leq \LM(A_f)$ then by~(iii) and~\eqref{int0f} we have $q\not= 0$ (and therefore either $f_2\not =0$ or $g_2\not=0$ or both) and
\begin{align*}
  \multideg(A_f f)&=\multideg(\LM(A_f)\LM(f_1)\LM(h)) \\
  &=\multideg(\LM(A_f))+\multideg(\LM(f_1))+\multideg(\LM(h))\\
  &\geq\multideg(\LM(g_1))+\multideg(\LM(f_1))+\multideg(\LM(h))\\
  &>^{\text{by }\eqref{eq:int1} \text{ if }f_2\not =0 \text{ and }\eqref{eq:int2} \text{ if }g_2\not =0} \multideg(f_2 \cdot g_1-f_1 \cdot g_2).
\end{align*}
The latter inequality contradicts (iii).

%%%%%%%%%%%%%%
From~\eqref{s_temp} and (i), it follows that
\small{
\begin{align}\label{eq:sext}
  S(f,g) & = C
  \left(\overbrace{(f_1-\LT(f_1)+A_g)}^{B_g}\cdot g+\overbrace{(\LT(g_1)-g_1+A_f)}^{B_f} \cdot f \right)+ C\cdot q|_{\{f,g\}}.
\end{align}
}
For $p\in \{f,g\}$, let $B_p$ be defined as in~\eqref{eq:sext}.
By Definition~\ref{def:reduction}, the claim follows from~\eqref{eq:sext} by
recalling that no term of $S(f,g)^*=C\cdot q|_{\{f,g\}}$ is divisible by any of $\LT(f),\LT(g)$ (see (ii)) and by showing that whenever $B_p p\not=0$ we have $\multideg(S(f,g))\geq \multideg(B_p p)$.
The latter follows by showing that $\LM(B_g\cdot g)\not = \LM(B_f\cdot f)$, whenever $B_g\cdot g\not=0$ and $B_f\cdot f\not=0$ (otherwise we are done). Indeed,
\begin{align*}
\LM(B_g\cdot g)&=\LM(h)\cdot \LM(g_1)\cdot \LM(f_1-\LT(f_1)+A_g),\\
\LM(B_f\cdot f)&=\LM(h)\cdot \LM(f_1)\cdot \LM(\LT(g_1)-g_1+A_f).
\end{align*}
By contradiction, if $\LM(B_g\cdot g) = \LM(B_f\cdot f)$ then 
\begin{align*}
\LM(f_1-\LT(f_1)+A_g)&=\frac{\LM(f_1)\cdot \LM(\LT(g_1)-g_1+A_f)}{\LM(g_1)}.
\end{align*}
 The latter is impossible because
 \begin{enumerate}
   \item $\LCM(\LM(f_1),\LM(g_1))=\LM(f_1)\cdot \LM(g_1)$ by \eqref{eq:int3};
   \item $\LM(g_1)>\LM(g_1-\LT(g_1)+A_f)$: this follows by noting that $\LM(g_1-\LT(g_1))<\LM(g_1)$ and  $\LM(A_f)<\LM(g_1)$ by~\eqref{c3}.
 \end{enumerate}
\end{proof}

Note that for any given pair of polynomials $f,g$ there could be several ways to decompose $f,g$ into a sum of 2 components still satisfying the conditions of Lemma~\ref{th:interlacing-lemma} and therefore yielding different normal forms. We will clarify how to apply it depending on the application. However, most of the time it will be ``natural'' since it will be applied to 2-terms polynomials and the 2 components (one possibly empty) of the input polynomials for the lemma are promptly identified.
%By Lemma~\ref{th:interlacing-lemma} we get the following simple yet useful corollary that will be used several times to bound the size of a \GB basis.
%
%\begin{corollary}\label{th:prime}
%  Given a finite set $G\subset \Field[\x]$, suppose that we have $h \cdot f,h \cdot g\in G$ such that
%%
%$
%\LCM(\LM(f),\LM(g))=\LM(f)\cdot \LM(g)
%$.
%%
%Then $S(h \cdot f,h \cdot g)\rightarrow_G 0$.
%\end{corollary}

%%%%%%%%%%%%%%%%%%%%%%%%%%%%%%%%%%%%%%%%%%%%%%%%%%%%%%%%%%%%%%%%%%%%%%%%%%%%%
%%%%%%%%%%%%%%%%%%%%%%%%%%%%%%%%%%%%%%%%%%%%%%%%%%%%%%%%%%%%%%%%%%%%%%%%%%%%%
\paragraph{Example: Set cover constraints.}\label{example-VC}
%%%%%%%%%%%%%%%%%%%%%%%%%%%%%%%%%%%%%%%%%%%%%%%%%%%%%%%%%%%%%%%%%%%%%%%%%%%%%
%%%%%%%%%%%%%%%%%%%%%%%%%%%%%%%%%%%%%%%%%%%%%%%%%%%%%%%%%%%%%%%%%%%%%%%%%%%%%

In \cite{Weitz17}, Weitz raised the question of effective derivation for problems without the strong symmetries discussed in his thesis~\cite{Weitz17}. As a starting example Weitz suggested the question whether the {\sc{vertex cover}} formulation~\eqref{eq:vc} for a given graph $(V,E)$ admits an effective derivation (see \cite[Chapter 6]{Weitz17}):
\begin{align}\label{eq:vc}
   \F_{VC}(V,E)=\{x_j^2-x_j\mid j\in V\}\cup \{(1-x_i)(1-x_j)\mid (i,j)\in E\}.
 \end{align}
We answer in the affirmative by showing that $\F_{VC}(V,E)$ admits the strongest effective derivation possible, namely it is a \GB basis, i.e. 1-effective for the vanishing ideal of the set of feasible solutions. (Actually in this paper we show this for two generalizations of \eqref{eq:vc}, namely {\sc set cover} and {\sc 2-sat}.)

Consider any $m\times n$ $0$-$1$ matrix $A$, and let $\F$ be the feasible region for the $0$-$1$ set covering problem defined by $A$:
\begin{align*}%\label{setcov}
\F=\{x\in \{0,1\}^n\mid Ax\geq e\},
\end{align*}
where $e$ is the vector of 1s.
 We denote by $A_i\subseteq \{1,\ldots,n\}$ the set of indices of nonzeros in the $i$-th row of $A$ (namely the \emph{support} of the $i$-th constraint).
 Let
 \begin{align}\label{eq:gb_cover}
   G & =\{x_j^2-x_j\mid j\in[n]\}\cup \{\prod_{j\in A_i}(1-x_j)\mid i\in[m]\}.
 \end{align}

\begin{proposition}
  Set~\eqref{eq:gb_cover} is a \GB basis for the vanishing ideal $\I(\F)$.
\end{proposition}
\begin{proof}
  By Theorem~\ref{th:crit}, set $G$, as defined in~\eqref{eq:gb_cover}, is a \GB basis for $\Ideal{\F}$ if and only if $S(f,g)\rightarrow_G 0$ for all distinct $f,g\in G$. The latter follows by using Lemma~\ref{th:interlacing-lemma} with $g_2=f_2=0$, $f=h\cdot f_1$, $g=h\cdot g_1$  and $h$ is the common factor $\prod_i(1-x_i)$ for $f$ and $g$ (possibly equal to 1).
\end{proof}

\section{Ternary \Majority\ Operation}\label{sect:majority}
%%%%%%%%%%%%%%%%%%
For Boolean languages there is only one $\Majority$ operation: $\Majority(x,y,z)$ is equal to $y$ if $y=z$, otherwise it is equal to $x$. It is known (see e.g.~\cite{Jeavons:1997}) that $\Majority$ closed Boolean relations of arbitrary arity are the relations definable by a formula in conjunctive normal form in which each conjunct contains at most two literals (also known as {\sc 2-Sat}).

%%%%%%%%%%%%%%%%%%%%%%%%%%%%%%%%%%%
It follows that any instance $\Cc=(\{x_1,\ldots,x_n\},\{0,1\},C)$ of $\CSP(\Gamma)$ (see Definition~\ref{def:csp}) whose polymorphism clone (see Definition~\ref{def:polymorph}) is closed under $\Majority$ can be easily and efficiently mapped to a set $F$ of polynomials of degree at most 2 such that: $\I_C=\GIdeal{F}$ and $Sol(\Cc)=\Variety{\I_\Cc}$ (see Section~\ref{sect:idealCSP}). Moreover, $\mathcal{B} \subseteq F\subseteq \mathcal{B}\cup \mathcal{Q}\cup \mathcal{L}\cup \mathcal{Z}$ where
%
% ``Boolean'' (denoted by $\mathcal{B}$), quadratic (denoted by $\mathcal{Q}$) and linear (denoted by $\mathcal{L}$) polynomials:
%
\begin{align*}
\mathcal{B}&=\{\pm(x_i^2-x_i)\mid i\in[n]\}, \quad \textsc{*Boolean*}\\
\mathcal{Q}&=\{\pm(x_i-\alpha)(x_j-\beta)\mid i,j\in[n], i\not=j, \alpha,\beta\in\{0,1\}\}, \quad \textsc{*Quadratic*}\\
\mathcal{L}&=\{\pm((\delta-\beta)x_i+(\gamma-\alpha)x_j+\alpha \beta-\gamma \delta) \mid i,j\in[n], i\not=j, \alpha,\beta,\gamma,\delta\in\{0,1\} \}, \  \textsc{*Linear*}\\
\mathcal{Z}&=\{\pm 1\}. \quad \textsc{*degree Zero*}
\end{align*}
By the weak Nullstellensatz (see Theorem~\ref{th:nullstz}), if $\mathcal{Z}\subseteq \I_\Cc$ then $\Cc$ is unsatisfiable. Moreover,
depending on the values of $\alpha,\beta,\gamma,\delta\in\{0,1\}$, note that for any $\ell\in \mathcal{L}$  we have that $\ell=0$ is equivalent to one of the following alternatives: $x_i+x_j=1$, $x_i=x_j$, $x_i=1$, $x_i=0$, $x_j=1$, $x_j=0$ or the zero polynomial.
It is easy to verify that set $\mathcal{B}\cup \mathcal{Q}\cup \mathcal{L}\cup \mathcal{Z}$ is 2-terms structured (see Definition~\ref{def:2terms}) with bivariate polynomials having degree at most 2. This set $F$ of 2-terms structured polynomials will be the input of Buchberger's Algorithm~\ref{Algo:buchbergerAlgo}.

The following lemma shows that set $\mathcal{B}\cup \mathcal{Q}\cup \mathcal{L}\cup \mathcal{Z}$ is closed under the multi-linearized polynomial division, namely for any $f,g\in\mathcal{B}\cup \mathcal{Q}\cup \mathcal{L}\cup \mathcal{Z}$ the remainder of the division of $f$ by $g$ and $\mathcal{B}$ is still in $\mathcal{B}\cup \mathcal{Q}\cup \mathcal{L}\cup \mathcal{Z}$ (recall that we are assuming grlex order).
\begin{lemma}\label{th:2satclosure}
For any $f,g\in\mathcal{B}\cup \mathcal{Q}\cup \mathcal{L}\cup \mathcal{Z}$ we have $f|_{\{g\}\cup \mathcal{B}} \in \mathcal{B}\cup \mathcal{Q}\cup \mathcal{L}\cup \mathcal{Z}$ and we say that set $\mathcal{B}\cup \mathcal{Q}\cup \mathcal{L}\cup \mathcal{Z}$ is \emph{closed under the multi-linearized polynomial division}.
\end{lemma}
\begin{proof}
%The proof is by simple inspection.
We will assume that $f,g\not\in \mathcal{Z}$ otherwise the claim is trivially true.
Then, the only interesting cases are when (a) $f\not \in \{g\}\cup \mathcal{B}$ (otherwise the remainder is zero) and (b) $f$ is divisible by $g$ (otherwise $f|_{\{g\}\cup \mathcal{B}}=f$ and the claim follows by the assumption).
It follows that $f\not\in \mathcal{B}$ and when $f\in \mathcal{L}$ then $g\not\in \mathcal{Q}$ (otherwise $f$ is not divisible by $g$ according to grlex order). Assuming (a) and (b), we distinguish between the following cases. We will assume w.l.o.g. that $f,g$ have been multiplied by appropriate constant to make $\LC(f)=\LC(g)=1$.
\begin{enumerate}
  \item $(f,g\in \mathcal{L})$. Then $f|_{\{g\}\cup \mathcal{B}}= {f|_{\{g\}}}$ and by (b) we have $\LM(f)=\LM(g)$, i.e. they have the same leading variable. It follows that ${f|_{\{g\}}}$ can be obtained from $f$ by eliminating the leading variable $\LM(f)$ according to the linear equation $g=0$. The resulting polynomial is in $\mathcal{L}\cup \mathcal{Z}$.
  \item $(f\in \mathcal{Q}\wedge g\in \mathcal{L})$. Hence, w.l.o.g., $f=(x_i-\alpha)(x_j-\beta)$ for some $i,j\in[n], i\not=j, \alpha,\beta\in\{0,1\}$ and $g=x_i+(a-b)x_k-a$ for some $a,b\in \{0,1\}$ with $x_i=\LM(g)$. Then $f|_{\{g\}} = ((b-a)x_k-(\alpha -a))(x_j-\beta)$. If $k\not =j$ then $f|_{\{g\}}=f|_{\{g\}\cup \mathcal{B}}$ and by simple inspection $f|_{\{g\}}\in \mathcal{Q}\cup \mathcal{L}$. Otherwise ($k=j$), we have $f|_{\{g\}\cup \mathcal{B}}\in \mathcal{B}\cup \mathcal{L}$.
  \item $(f,g\in \mathcal{Q})$. Then, let $f=(x_i-\alpha)(x_j-\beta)$ and $g=(x_i-\gamma)(x_j-\delta)$, for some $i,j\in[n], i\not=j,\alpha,\beta,\gamma,\delta\in \{0,1\}$
     (note that by (b) we are assuming $\LM(f)=\LM(g)$, so they have the same variables). It follows that $f|_{\{g\}\cup \mathcal{B}}=f|_{\{g\}}=f-g= (\delta-\beta)x_i+(\gamma-\alpha)x_j+\alpha \beta-\gamma \delta\in \mathcal{L}$.
\end{enumerate}
\end{proof}

%
%At the beginning of Algorithm~\ref{Algo:buchbergerAlgo} we set $G=F=\mathcal{B}\cup \mathcal{Q}$.
%If for any two polynomials $f,g\in G$ we have $S(f,g)\rightarrow_G 0$ then we are done by Proposition~\ref{th:crit}, otherwise $G$ has to be extended by adding ${{S(f,g)}|_{G}}$ (see line~\ref{eq:addspoly}). As proved in the following, the extension of $G$ can be obtained by adding to $G$ a reduced $S$-polynomial ${{S(f,g)}|_{G}}$ that is either equal to $q_k=(x_i-\alpha)(x_j-\beta)$, or equal to $l_k=(\delta-\beta)x_i+(\gamma-\alpha)x_j+\alpha \beta-\gamma \delta$ for some $i,j\in[n]$, $\alpha,\beta,\gamma,\delta\in\{0,1\}$.
The next lemma shows that set $\mathcal{B}\cup \mathcal{Q}\cup \mathcal{L}\cup \mathcal{Z}$ is also closed under another important operation, namely the multi-linearized $S(f,g)^*$-polynomial (see~\eqref{eq:intnormal}). Note that the multi-linearized version of $S(f,g)^*$ is equal to $S(f,g)^*|_{\mathcal{B}}$. This operation will be crucially employed within Buchberger's Algorithm~\ref{Algo:buchbergerAlgo} to get the claimed results.
\begin{lemma}\label{th:GB2sat}
For any $f,g\in\mathcal{B}\cup \mathcal{Q}\cup \mathcal{L}\cup \mathcal{Z}$ we have $S(f,g)^*|_{\mathcal{B}} \in \mathcal{B}\cup \mathcal{Q}\cup \mathcal{L}\cup \mathcal{Z}$, namely set $\mathcal{B}\cup \mathcal{Q}\cup \mathcal{L}\cup \mathcal{Z}$ is closed under $S(f,g)^*|_{\mathcal{B}}$-polynomial composition.
\end{lemma}
%%%%%%%%%%%%%%%%%%%%%%%%%%%%%%%%%%%%%%%%%%%%%%%%%%
\begin{proof}%[Proof of Lemma~\ref{th:GB2sat}.]
For simplicity, we assume w.l.o.g. that $f,g$ have been multiplied by appropriate constant to make $\LC(f)=\LC(g)=1$.
For any given $f,g\in\mathcal{B}\cup \mathcal{Q}\cup \mathcal{L}\cup \mathcal{Z}$  we distinguish between the following complementary cases (we assume non zero polynomials with $f\not = g$ otherwise $S(f,g)=S(f,g)^*=0$):
\begin{enumerate}
%%%%%%%%%%%%%%%%%%%%%%%%%%%%%%%%%%%%%%%%%%%%%%%%%%%%%%%%%%%%%%%%%%%%%%%%%
\item If $\LCM(\LM(f),\LM(g))=\LM(f)\cdot \LM(g)$ then we have $S(f,g)^*=0$ (by Lemma~\ref{th:interlacing-lemma} with $f_1=f$ and $g_1=g$).\label{en:0}
%%%%%%%%%%%%%%%%%%%%%%%%%%%%%%%%%%%%%%%%%%%%%%%%%%%%%%%%%%%%%%%%
\item Else if $g\in \mathcal{L}$, ($f\in \mathcal{L}$ is symmetric) then we claim that $S(f,g)^*=f|_{\{g\}}$, hence $S(f,g)^*|_{\mathcal{B}}=f|_{\{g\}\cup\mathcal{B}}$, and the claim follows by Lemma~\ref{th:2satclosure}.

    Indeed, if both $f,g\in \mathcal{L}$ then by the assumptions (recall we are not in Case~\ref{en:0}) we have $S(f,g)=f-g= {f|_{\{g\}}}$. Note that $S(f,g)$ is not divisible by $f,g$ which implies $S(f,g)^*=S(f,g)$.

    Otherwise, assume $f\in \mathcal{B}\cup \mathcal{Q}$ (and $g\in \mathcal{L}$).
    Then, recall that we are assuming that $\LC(g)=1$ and $f$ is divisible by $g$, otherwise we are in Case~\ref{en:0}, which implies that $\LM(g)$ is a variable, say $x_i$, that appears in $f$ as well. Then, w.l.o.g., we can write $f,g$ as follows (the labels over the different parts of $f,g$ will be used while applying Lemma~\ref{th:interlacing-lemma}):
    \begin{align*}
      f&=\overbrace{(x_i-\alpha)}^h\overbrace{(x_k-\beta)}^{f_1}, \text{ for some } k\in[n]  \text{ and }  \alpha,\beta\in\{0,1\}; \\
      g&= \overbrace{(x_i-\alpha)}^h + \overbrace{\alpha+(\gamma-\delta)x_j -\gamma}^{g_2}, \text{ for some }j\in[n]\setminus \{i\} \text{ and } \gamma,\delta\in\{0,1\}.
    \end{align*}
By applying Lemma~\ref{th:interlacing-lemma} (with $h,f_1,g_2$ as above and $f_2=0$ and $g_1=1$) we have
    $$S(f,g)^*=-f_1g_2|_{\{f,g\}}=f|_{\{g\}},$$
where the latter follows by noting that $f|_{\{g\}}=-f_1g_2$ and therefore not divisible neither by $f$ nor by $g$.

\item Else if $f,g\in \mathcal{B}\cup\mathcal{Q}$ and $\LM(f),\LM(g)$ share exactly one variable, i.e. $\LM(f)=x_i\cdot x_j$ and $\LM(g)=x_i\cdot x_k$ for some $i,j,k\in [n]$ with $j\not=k$ (but we can have $j=i$ xor $k=i$).
    We distinguish between the following complementary subcases:
\begin{itemize}
\item $f,g\in \mathcal{B}\cup\mathcal{Q}$ and they ``agree'' on the shared variable, i.e.
    \begin{align*}
      f&=\overbrace{(x_i-\alpha)}^{h}\overbrace{(x_j-\beta)}^{f_1}, \\
      g&=\overbrace{(x_i-\alpha)}^{h}\overbrace{(x_k-\gamma)}^{g_1},
    \end{align*}
    for some $\alpha,\beta,\gamma\in \{0,1\}$ (and $k\not=j$). In this case by applying Lemma~\ref{th:interlacing-lemma} (with $h,f_1,g_1$ as above and $g_2=f_2=0$), we have $S(f, g)^*= 0$.
\item Otherwise, assume $f=x_i(x_j-\beta)$ and $g=(x_i-1)(x_k-\gamma)$, where $\beta,\gamma\in \{0,1\}, k\not=j$  (remaining cases are symmetric). Apply Lemma~\ref{th:interlacing-lemma} with
    \begin{align*}
      f&=\overbrace{(x_i-1)}^{h}\overbrace{(x_j-\beta)}^{f_1}+\overbrace{(x_j-\beta)}^{f_2}, \\
      g&=\overbrace{(x_i-1)}^{h}\overbrace{(x_k-\gamma)}^{g_1}.
    \end{align*}
    It follows that $S(f,g)^*=(x_j-\beta)(x_k-\gamma)\in \mathcal{Q}$ and $S(f,g)^*=S(f,g)^*|_{\mathcal{B}}$.
\end{itemize}

\item Else if $f,g\in \mathcal{B}\cup\mathcal{Q}$ and $\LM(g)=\LM(f)$:
in this case $f$ and $g$ have the same variables $x_i,x_j$, for some $i,j\in[n], i\not=j$ and $f,g\in \mathcal{Q}$ (the latter because we are assuming $f\not=g$ and $\LM(g)=\LM(f)$, so it cannot happen that $f\in \mathcal{B}$ or $g\in \mathcal{B}$). Then $f=(x_i-\alpha)(x_j-\beta)$ and $g=(x_i-\gamma)(x_j-\delta)$, for some $\alpha,\beta,\gamma,\delta\in \{0,1\}$. Note that $S(f,g)=f-g= (\delta-\beta)x_i+(\gamma-\alpha)x_j+\alpha \beta-\gamma \delta$.
$S(f,g)$ is not divisible by $f,g$ which implies $S(f,g)^*=S(f,g)$ and therefore $S(f,g)^*\in \mathcal{L}$  and $S(f,g)^*=S(f,g)^*|_{\mathcal{B}}$.
%%%%%%%%%%%%%%%%%%%%%%%%%%%%%%%%%%%%%%%%%%%%%%%%%%%%%%%%%%%%%%%%%%%%%%%%%
%%%%%%%%%%%%%%%%%%%%%%%%%%%%%%%%%%%%%%%%%%%%%%%%%%%%%%%%%%%%%%%%%%%%%%%%%
\end{enumerate}
\end{proof}
%%%%%%%%%%%%%%%%%%%%%%%%%%%%%%%%%%%%%%%%%%%%%%%%%%%%%%%%%%%%%%%%%%%%%%%%%%%%
\begin{lemma}
For any given $\emph{\CSP}(\Gamma)$ instance $\Cc$, if $\Majority\in\Pol(\Gamma)$ then the reduced \GB basis for the combinatorial ideal $\I_\Cc$ is computable in $n^{O(1)}$ time and it is 2-terms structured.
\end{lemma}
%%%%%%%%%%%%%
\begin{proof}
We compute a \GB basis  $G$ for $\I_\Cc$ by using Buchberger's Algorithm~\ref{Algo:buchbergerAlgo} with the following change: at line~\ref{eq:reminder} of Algorithm~\ref{Algo:buchbergerAlgo}, replace $S(f,g)$ with $S(f,g)^*|_{\mathcal{B}}$ (see~\eqref{eq:intnormal}) and then reduce it modulo $G$, i.e. divide $S(f,g)^*|_{\mathcal{B}}$ by $G$ in any order and return the remainder that we denote by $S(f,g)^*|_{G}$. Note that $S(f,g)^*|_{G}$ is a normal form of $S(f,g)$ by $G$, namely there is an ordering of the polynomials division that make $S(f,g)^*|_{G}=S(f,g)|_G$.
 Therefore Algorithm~\ref{Algo:buchbergerAlgo} with the above specified changes returns a \GB basis, since Buchberger's Algorithm is guaranteed to return a \GB basis independently on the order in which we perform polynomial divisions at line~\ref{eq:reminder}. Moreover, we claim that the returned \GB basis will be a subset of $\mathcal{B}\cup \mathcal{Q}\cup \mathcal{L}\cup \mathcal{Z}$.

At the beginning of the Buchberger's Algorithm, line~\ref{eq:startingG} of Algorithm~\ref{Algo:buchbergerAlgo}, we have $G=F\subseteq\mathcal{B}\cup \mathcal{Q}\cup \mathcal{L}\cup \mathcal{Z}$, where $F$ is the set of polynomials defined at the beginning of this section.
Lemma~\ref{th:2satclosure} and Lemma~\ref{th:GB2sat} show that $\mathcal{B}\cup \mathcal{Q}\cup \mathcal{L}\cup \mathcal{Z}$ is closed under Boolean polynomial division and under $S(f,g)^*|_{\mathcal{B}}$-polynomial composition for any $f,g\in \mathcal{B}\cup \mathcal{Q}\cup \mathcal{L}\cup \mathcal{Z}$.
It follows that the condition at line~\ref{eq:Buch_cond} of Algorithm~\ref{Algo:buchbergerAlgo} (i.e. $S(f,g)^*|_G\not = 0$) is satisfied at most $O(n^2)$ times, since $|\mathcal{B}\cup \mathcal{Q}\cup \mathcal{L}\cup \mathcal{Z}|=O(n^2)$. Therefore, after at most $O(n^2)$ many times condition at line~\ref{eq:Buch_cond} of Algorithm~\ref{Algo:buchbergerAlgo} is satisfied, we have $S(f,g)\rightarrow_G 0$ for any $f,g\in G$. By Theorem~\ref{th:crit}, this implies that a \GB basis for $\Majority$ closed languages can be computed in polynomial time.

Finally, for any fixed monomial ordering the (unique) reduced \GB basis can be obtained from a non reduced one $G$ by repeatedly dividing each element $g\in G$ by $G\setminus \{g\}$. Lemma~\ref{th:2satclosure} implies that it is 2-terms structured.
\end{proof}
%%%%%%%%%%%%%%%%%%%%%%%%%%%%%%%%%%%%%%%%%%%%%%%%%%%%%%%%%%%%%%%%

%%%%%%%%%%%%%%%%%%%%%%%%%%%%%%%%%%%%%%%%%%%%%%%%%%%%%%%%%%%%%%%%%%%%%%%%%%%%%%%
%%%%%%%%%%%%%%%%%%%%%%%%%%%%%%%%%%%%%%%%%%%%%%%%%%%%%%%%%%%%%%%%%%%%%%%%%%%%%%%
\section{Binary \Max\  and \Min\  Operations}\label{sect:min_max}
%%%%%%%%%%%%%%%%%%
For Boolean languages there are only two idempotent binary operations (which are not projections) corresponding to the $\Max$ operation (logical \emph{OR}) and the $\Min$ operation (logical \emph{AND}).

It is known (see e.g.~\cite{Jeavons:1997} and the references therein) that a Boolean relation is closed under $\Max$ operation if and only if can be defined by a conjunction of clauses each of which contains at most one negated literal (also known as \emph{dual-Horn} clauses).
%
%%%%%%%%%%%%%%%%%%%%%%%%%%%%%%%%%%%
It follows that any instance $\Cc=(\{x_1,\dots,x_n\},\{0,1\},C)$ of $\CSP(\Gamma)$ whose polymorphism clone is closed under the $\Max$ operation can be mapped to a set $F$ of 2-terms polynomials (see Definition~\ref{def:2terms}) such that: $\I_C=\GIdeal{F}$, $Sol(\Cc)=\Variety{\I_\Cc}$ (see Section~\ref{sect:idealCSP}) and $F\subseteq \T^-$. Indeed, every clause with variables in $S\subseteq [n]$ and at most one negated literal can be represented by the following system of (negative) 2-terms polynomials equalities:
\begin{align*}
 &x_i^2-x_i=0, &\qquad i\in S;\\%\label{eq:hornktype1}\\
  &\prod_{j\in S} (x_j-1)+\alpha\prod_{j\in S\setminus\{i\}} (x_j-1)=0, &\qquad \text{for some }\alpha \in \{0, 1\}, i\in S. %\label{eq:hornktype3}
\end{align*}
%%%%%%%%%%%%%%%%%%%%%%%%%%%%%%%%%%

Similarly, a Boolean relation is closed under $\Min$ operation if and only if can be defined by a conjunction of clauses each of which contains at most one unnegated literal (also known as \emph{Horn} clauses). Any instance $\Cc=(X,\{0,1\},C)$ of $\CSP(\Gamma)$ whose polymorphism clone is closed under the $\Min$ operation can be mapped to an equivalent set $F\subseteq\T^+$ of 2-terms polynomials.

The following lemma shows that set $\T^-$ (or $\T^+$) is closed under polynomial division, namely for any $f,g\in \T^-$ (or $\in\T^+$) the remainder of the division of $f$ by $g$ is still in $\T^-$ (or $\T^+$) (recall that we are assuming grlex order).
%%%%%%%%%%%%%%%%%%%%%%%%%%%%%%%%%%%%%%%%%%%%%%%%%%%%%%%%%%%%%%%%%%%%%%%%%%%%
\begin{lemma}\label{th:Max_Mindivclosure}
For any $f,g\in\T^-$ (or in $\T^+$) we have $f|_{\{g\}} \in \T^-$ (or in $\T^+$) and we say that set $\T^-$ (or $\T^+$) is \emph{closed under polynomial division}.
\end{lemma}
%%%%%%%%%%%%
\begin{proof}
We show the proof for set $\T^-$, the other case is symmetric.

Consider any $f,g\in \T^-$.
For simplicity, we assume w.l.o.g., that $f,g$ have been multiplied by appropriate constant to make $\LC(f)=\LC(g)=1$. Note that the only cases where $f|_{\{g\}}\not =f$ (otherwise we are done) is when
\begin{align*}
  f&=\prod_{i\in B_1}(x_i-1)\prod_{j\in A_1}(x_j-1)+\alpha \prod_{i\in A_2}(x_i-1), \\
  g&=\prod_{i\in B_1}(x_i-1)+\beta \prod_{i\in B_2}(x_i-1),
\end{align*}
for some $A_1,A_2,B_1,B_2\subseteq [n]$, $\alpha,\beta\in \{0,\pm 1\}$, $B_1\not=\emptyset$ and $\LM(f)=\prod_{i\in A_1\cup B_1} x_i$, $\LM(g)=\prod_{i\in B_1} x_i$. Then, $f|_{\{g\}}=-\beta\prod_{i\in B_2}(x_i-1)\prod_{j\in A_1}(x_j-1)+\alpha \prod_{i\in A_2}(x_i-1)$ and the claim follows.
\end{proof}

The next lemma shows that set $\T^-$ (or $\T^+$) is also closed under the $S(f,g)^*$-polynomial composition (see~\eqref{eq:intnormal}). This, by using Lemma~\ref{th:Max_Mindivclosure}, will imply that the reduced \GB basis is a subset of $\T^-$ ($\T^+$).

%%%%%%%%%%%%%%%%%%%%%%%%%%%%%%%%%%%%%
\begin{lemma}\label{th:max_2terms}
For any given $\emph{\CSP}(\Gamma)$ instance $\Cc$, if $\Max\in\Pol(\Gamma)$ (or $\Min\in\Pol(\Gamma)$) then the reduced \GB basis for the combinatorial ideal $\I_\Cc$ is a subset of $\T^-$ ($\T^+$).
\end{lemma}
%%%%%%%%%%%%%%%%%%%%%%%%%%%%%%%%%%%%%
\begin{proof}
We show the claim when $\Max\in\Pol(\Gamma)$; the other case has a similar but simpler proof, since the use of the ``standard'' $S$-polynomials suffices in the corresponding arguments below.

By Lemma~\ref{th:Max_Mindivclosure}, set $\T^-$ is closed under polynomial division.
Next we observe that for any $f,g\in\T^-$ we have $S(f,g)^* \in \T^-$, namely set $\T^-$ is closed under $S(f,g)^*$-polynomial composition (see~\eqref{eq:intnormal}).
Indeed for any $f,g\in \T^-$, if $S(f,g)^* \not= 0$ (otherwise we are done) then the claim follows by applying Lemma~\ref{th:interlacing-lemma} with $hf_1$ (of $hg_1$) be equal to the (negative) term of $f$ (of $g$) with the highest multidegree.

As already observed at the beginning of this section, any instance $\Cc$ of $\CSP(\Gamma)$ whose polymorphism clone is closed under the $\Max$ operation can be mapped to a set $F\subseteq \T^-$ such that: $\I_C=\GIdeal{F}$, $Sol(\Cc)=\Variety{\I_\Cc}$.

Consider Buchberger's algorithm (see Algorithm~\ref{Algo:buchbergerAlgo}) with set $F\subseteq \T^-$ as input and with the following change: at line~\ref{eq:reminder} of Algorithm~\ref{Algo:buchbergerAlgo}, replace $S(f,g)$ with $S(f,g)^*$ (see~\eqref{eq:intnormal}). Then, every set $G$ considered in Algorithm~\ref{Algo:buchbergerAlgo} is a subset of $\T^-$. Since Algorithm~\ref{Algo:buchbergerAlgo} is guaranteed to return a \GB basis (for any chosen normal form $S(f,g)|_G$), it follows that there exists a \GB basis $G$ that is a subset of $\T^-$.

Finally, the (unique) reduced \GB basis can be obtained from a non reduced one $G$ by repeatedly dividing each element $g\in G$ by $G\setminus \{g\}$. By Lemma~\ref{th:Max_Mindivclosure}, it follows that the unique \GB basis is  a subset of $\T^-$.
\end{proof}

\begin{remark}
  Contrary to what happened to $\Majority$ closed languages, the following example seems to suggest that the reduced \GB basis for $\Max$ ($\Min$) closed language problems might have arbitrarily large degree.
  Let $n$ be an odd number. Consider the following generating set $$F=\{f_k\mydef x_k(x_{k+1}-1)(x_{k+2}-1)\mid k\leq n-2 \text{ and $k$ odd}\}\cup \{x_i^2-x_i\mid i\in [n]\}.$$
 Note that the degree of each polynomial from $F$ is at most 3, and $F$ is the generating set of a combinatorial ideal $\I_\Cc$ corresponding to an instance $\Cc\in \CSP(\Gamma)$ with $\Max\in\Pol(\Gamma)$ (this is an instance of dual-Horn 3-\textsc{sat}).
 We experimentally noticed (for every odd $n\leq 19$) that the reduced \GB basis has polynomials of degree up to $(n-1)/2+2$.

 We leave as an open problem to determine the size of the reduced \GB bases for $\Max$ (or $\Min$) closed language problems. We conjecture their sizes to be superpolynomial in the number of variables in the worst case.
\end{remark}

Computing a \GB basis is certainly a sufficient condition for membership testing, but not strictly necessary. In Section~\ref{sect:truncatedGB} we show how to efficiently resolve the membership question without computing a full \GB basis, but a \emph{truncated} one. As already remarked, this is considerably different from the  bounded degree version of Buchberger's algorithm considered in~\cite{CleggEI96}.
%In Section~\ref{sect:sparse} we discuss the case of  \emph{sparse} input polynomials.

%%%%%%%%%%%%%%%%%%%%%%%%%%%%%%%
\subsection{Truncated \GB bases}\label{sect:truncatedGB}
%%%%%%%%%%%%%%%%%%%%%%%%%%%%%%%
Assuming $\Max\in\Pol(\Gamma)$ (or $\Min\in\Pol(\Gamma)$),
we want to test whether a given polynomial $f$ of degree $d$ lies in the combinatorial ideal $\I_\Cc$ corresponding to a given $\CSP(\Gamma)$ instance $\Cc$.
As already observed (see Section~\ref{sect:tract}), the membership test can be efficiently computed by using polynomials from the \emph{truncated} reduced \GB basis $G_d= G\cap \Field[x_1,\ldots,x_n]_d$, where $G$ is the reduced \GB basis for $\I_\Cc$.
Below we show how to compute $G_d$ in $n^{O(d+1)}$ time, for any degree $d\in \N$.
This yields an efficient algorithm for the membership problem (this is ``efficient'' because the size of the input polynomial $f$ is $n^{O(d)}$). ``Sparse'' polynomials, i.e. polynomials with ``few'' terms, are discussed in Section~\ref{sect:sparse}.

Note that the computation in $n^{O(d+1)}$ time of the truncated \GB basis is sufficient for efficiently computing Theta Bodies SDP relaxations and bound the bit complexity of $\sos$ for this class of problems (for these applications $d=O(1)$).

\begin{lemma}\label{th:Gd}
If $\Max\in\Pol(\Gamma)$ then the truncated \GB basis $G_d$ for the combinatorial ideal $\I_\Cc$ corresponding to a given $\emph{\CSP}(\Gamma)$ instance $\Cc$ can be computed in $n^{O(d+1)}$ time for any $d\in \N$.
\end{lemma}
\begin{proof}
By Lemma~\ref{th:max_2terms}, $G_d\subseteq \T^-_d\mydef \T^-\cap \Field[x_1,\ldots,x_n]_d$.
  Note that in $\T^-_d$ there are $n^{O(d)}$ polynomials of degree $\leq d$, each with $O(d)$ variables.
  For any given $p \in \T^-_d$ we can check whether $p\in\I_\Cc$ as follows.

  By the Strong Nullstellensatz~\eqref{eq:strong_nstz} and the radicality~\eqref{eq:ICradical} of $\I_\Cc$ (see~\eqref{eq:IC}),
an equivalent way to solve the membership problem $p\in \I_\Cc$ is to answer the following question:
\begin{align}\label{eq:gboracle}
  \text{Does it exist }\tilde{x}\in \{0,1\}^n \text{ such that } \left(p(\tilde{x})\not=0 \wedge \tilde{x}\in \Variety{\I_\Cc}\right) \text{ ? }
\end{align}
%\end{mdframed}
%
Note that the answer to Question~\eqref{eq:gboracle} is affirmative if and only if $p\not\in \Ideal{\Variety{\I_\Cc}}$ and therefore $p\not \in \I_\Cc$ by~\eqref{eq:strong_nstz} and~\eqref{eq:ICradical}.

Let $X_p$ be the set of variables appearing in $p$. Consider a subset $Y\subseteq X_p$ and a mapping $\phi: Y\rightarrow \{0,1\}$. We say that $(Y,\phi)$ is a \emph{non-vanishing partial assignment} of $p$ if there exists no assignment of the variables in $X_p\setminus Y$ that makes $p$ equal to zero while $\{x_i=\phi(x_i):i\in Y\}$; moreover, $(Y,\phi)$ is \emph{minimal} with respect to set inclusion if by removing any variable $x_j$ from $Y$ there is an assignment of the variables in $X_p\setminus (Y\setminus\{x_j\})$ that makes $p$ equal to zero while $\{x_i=\phi(x_i):i\in Y\setminus\{x_j\}\}$.

For each $p\in \T^-_d$ there are $O(d)$ minimal non-vanishing partial assignments. These correspond to minimal partial assignments that make the 2-terms sum in $p$ not zero: for example one term of $p$ equal to $1$, so all the variables in that term are set to zero, and the other term being equal to zero (or 1), so one variable in the other term is set to one (or all variables set to zero, depending on $p$). Note that for each $\tilde{x}\in \{0,1\}^n$ such that $p(\tilde{x})\not=0$ there is a minimal non-vanishing partial assignments. It follows that we can answer to question~\eqref{eq:gboracle} by simply checking for each minimal non-vanishing partial assignment for $p$ if it extends to a feasible solution for $\Cc$. The latter can be checked in polynomial time since $\Gamma$ is an idempotent constraint language, namely it contains all singleton unary
relations.
\end{proof}
Similarly as for $\Max$-closed languages we can obtain the following result.
\begin{lemma}\label{th:Gd_min}
If $\Min\in\Pol(\Gamma)$ then the truncated reduced \GB basis $G_d$ for the combinatorial ideal $\I_\Cc$ corresponding to a given $\emph{\CSP}(\Gamma)$ instance $\Cc$ can be computed in $n^{O(d+1)}$ time for any $d\in \N$.
\end{lemma}

%%%%%%%%%%%%%%%%%%%%%%%%%%%%%%%%%%%%%%%%%%%%%%%%%%
\subsection{Sparse Polynomials}\label{sect:sparse}
%%%%%%%%%%%%%%%%%%%%%%%%%%%%%%%%%%%%%%%%%%%%%%%%%%
We call a polynomial positive (negative) \emph{$k$-sparse} if it can be
represented by $k$ positive (negative) terms (see Definition~\ref{def:2terms}) with nonzero coefficients.

In the following, we discuss $\Min$-closed languages and complement Lemma~\ref{th:Gd_min} by considering the membership problem for positive $k$-sparse polynomials of degree $d$. Note that positive $k$-sparse polynomials means polynomials with at most $k$ monomials with nonzero coefficients. A symmetric argument holds for $\Max$-closed languages by replacing positive terms with negative terms.

 When $k\ll n^{O(d)}$, we show that we can remove the exponential dependance on $d$ by (i) either efficiently compute a polynomial subset (that depends on $f$) $G_f\subseteq \I_\Cc$ such that $\reduce f {G_f}=0$, (ii) or show a certificate that $f\not\in \I_\Cc$.

\begin{lemma}
If $\Min\in\Pol(\Gamma)$ then we can test in $(kn)^{O(1)}$ time whether a given positive $k$-sparse polynomial $p$ lies in the combinatorial ideal $\I_\Cc$ of a given $\CSP(\Gamma)$ instance $\Cc$.
\end{lemma}
\begin{proof}
Let $G$ be the reduced \GB basis of $\I_\Cc$ (according to grlex order). By Lemma~\ref{th:max_2terms} we know that $G\subseteq \T^+$.
   We assume w.l.o.g. that $p$ is multilinear, otherwise we denote by $p$ the remainder of $p$ divided by $\{x_i^2-x_i:i\in[n]\}$.

   If $p\in \I_\Cc$ then there exists a (finite) set of (positive) 2-terms polynomials $\{g_1,\ldots,g_\ell\}\subseteq G$ such that
  $p=\sum_{i=1}^\ell g_i \cdot q_i$ where $q_i\in \Field[x_1,\ldots,x_n]$ and $\multideg(g_i \cdot q_i)\leq \multideg(p)$ (see Lemma~\ref{th:gbprop}).

  Each $g_i\cdot q_i$ is a (weighted) sum of positive 2-terms polynomials ($q_i$ is a weighted sum of monomials and $g_i$ times any weighted monomial is a weighted positive 2-terms polynomial from $\I_\Cc$). It follows that $p=\sum_{t\in S} c_t\cdot  t$ for some $S\subseteq \T^+\cap\I_\Cc$ and $c_t\in\Field$. Each $t\in S$ can be written as $t=t_a+t_b$, where $t_a,t_b$ are two positive terms and $p=\sum_{t\in S} c_t \cdot(t_{a}+t_{b})$.
We start observing the following simple argument.
  If there are two (not equal) polynomials from $S$, say $u,t\in S$ such that $u+t=t_a+u_b\in \T^+\cap\I_\Cc$ then we have $c_u \cdot u+ c_t \cdot t\in \I_\Cc$ and
  $$c_u \cdot u+ c_t \cdot t=(c_u-c_t) \cdot u+ c_t \cdot (u+t)=(c_u-c_t) \cdot \overbrace{u}^{\text{in } \T^+\cap\I_\Cc}+ c_t \cdot \overbrace{(t_a+ u_b)}^{\text{in } \T^+\cap\I_\Cc}.$$

We are assuming that $p$ is $k$-sparse, therefore $p=\sum_{i=1}^k w_{i}\cdot \mu_i$ for some $w_i\in \Field$ and $\mu_i$ monomials in $\Field[x_1,\ldots,x_n]$.
From the example above it is easy to argue that if $p\in \I_\Cc$ then there exists a pair of monomials $\mu_i$ and $\mu_j$ such that $\mu_i+\alpha \mu_j\in \I_\Cc$, for some $\alpha\in \{0,\pm 1\}$. For each pair $\mu_i,\mu_j$ of monomials we can check in polynomial time whether $\mu_i+\alpha \mu_j\in \I_\Cc$ for some $\alpha\in \{0,\pm 1\}$ (the polynomial time algorithm is similar to the one described in the proof of Lemma~\ref{th:Gd}). If none of the algebraic sums of pairs is in $\I_\Cc$ then we can conclude that $p\not \in \I_\Cc$. Otherwise, if $\mu_i+\alpha \mu_j\in \I_\Cc$ for some $i,j$ and $\alpha\in \{0,\pm 1\}$ then if $p\in \I_\Cc$ then also $p-w_i(\mu_i+\alpha\mu_j)\in \I_\Cc$. In the latter case we apply the same arguments to $p-w_i(\mu_i+\alpha\mu_j)$ but now the new polynomial has one monomial less, so in at most $k$ times either we conclude that $p\in \I_\Cc$ or $p\not\in \I_\Cc$.
\end{proof}

%%%%%%%%%%%%%%%%%%%%%%%%%%%%%%%%%%%%%%%%%%%%%%%%%%%%%%%%%%%%%%%%%%%%%%%%%%%%%%%%%%%
%%%%%%%%%%%%%%%%%%%%%%%%%%%%%%%%%%%%%%%%%%%%%%%%%%%%%%%%%%%%%%%%%%%%%%%%%%%%%%%%%%%
%%%%%%%%%%%%%%%%%%%%%%%%%%%%%%%%%%%%%%%%%%%%%%%%%%%%%%%%%%%%%%%%%%%%%%%%%%%%%%%%%%%
%%%%%%%%%%%%%%%%%%%%%%%%%%

\section{The Ideal Membership Problem: Intractability}\label{sect:necessary}
%%%%%%%%%%%%%%%%%%%%%%%%%%%%%%%%%%%%%%%%%%%%%%%%%%%%%%%%%%%%%%%%%%%%%%%%%%%%%%%
In this section we provide the proof of Theorem~\ref{th:hardness}. Let us start by giving the following definition.
\begin{definition}\label{def:unary}
  We say that an operation $f: D^k\to D$ is \emph{essentially unary} if there exists a coordinate $i\in\{1,\ldots,k\}$ and a unary operation $g:D\to D$ such that $f(d_1,\ldots,d_k)=g(d_i)$ for all values $d_1,\ldots,d_k\in D$. If in addition it happens that $g$ is bijective then we say that $f$ \emph{acts as a permutation}.
\end{definition}
A result by Post in 1941 (see e.g. \cite[Theorem 5.1]{Chen09}) says that a clone over $\{0,1\}$ either contains only essentially unary operations, or contains one of the the following four operations $\{\Majority, \Minority, \Min, \Max\}$.
Observe that the only essentially unary operations on $\{0,1\}$ are the two constant operations and the two permutations (the identity and $\neg$).   

With this in place we are ready to prove Theorem~\ref{th:hardness}.
For convenience we repeat the statement of Theorem~\ref{th:hardness} below.
%
%%%%%%%%%%%%%%%%%%%%%%%%%%%%%%%%%%%%%%%%%%%%%%
\begin{lemma}\label{th:necessary}
  Let $\Gamma_1$ be a Boolean language with the solution space of every constraint closed under one constant operation $c \in\{0,1\}$.
  Let $\Gamma_2$ be a Boolean language with the solution space of every constraint closed under both constant operations $1$ and $0$. Assume that $\Gamma_1$ and $\Gamma_2$ are not simultaneously closed under any of $\Majority$, $\Minority$, $\Min$ or $\Max$ operations. Then, for $i\in\{1,2\}$, the problem $\IMP_i(\Gamma_i)$ is coNP-complete. If the constraint language $\Gamma_2$ has the operation $\neg$ as a polymorphism then the problem $\IMP_1(\Gamma_2)$ is coNP-complete.
\end{lemma}
%%%%%%%%%%%%%%%%%%%%%%%%%%%%%%%%%%%%%%%%%%%%%
\begin{proof}
The \emph{singleton expansion} of language $\Gamma_1$ is the language $\Lambda_1=\Gamma_1\cup\{(1-c)\}$. Let $\Lambda_2=\Gamma_2\cup\{(0)\}\cup\{(1)\}$ be the constraint language obtained by augmenting $\Gamma_2$ with both unary relations of size one. By Theorem~\ref{th:schaefer} and the assumptions on $\Gamma_i$, $\CSP(\Lambda_i)$ is NP-complete, for $i\in \{1,2\}$. We show that $\CSP(\Lambda_i)$ polynomial-time reduces to not-$\IMP_i(\Gamma_i)$, for $i\in \{1,2\}$.

\begin{enumerate}
    %%%%%%%%%%%%%%%%%%%%%%%%%%%%%%%%%%%%%%%%%%%%%%%%%%%%%%%%%%%%%%%%%%%%%%%%%%%%%%%%%%%%
    %%%%%%%%%%%%%%%%%%%%%%%%%%%%%%%%%%%%%%%%%%%%%%%%%%%%%%%%%%%%%%%%%%%%%%%%%%%%%%%%%%%%
    \item In the following we show that $\CSP(\Lambda_1)$ polynomial-time reduces to not-$\IMP_1(\Gamma_1)$. The latter implies that $\IMP_1(\Gamma_1)$ is coNP-complete as claimed.
    
    Let $\Cc_{\Lambda_1}=(X,\{0,1\},C_{\Lambda_1})$ be a given instance of $\CSP(\Lambda_1)$ and let $C_{\Gamma_1}\subseteq C_{\Lambda_1}$ be the maximal set of constraints from $C_{\Lambda_1}$ over the language $\Gamma_1$. Instance $\Cc_{\Lambda_1}$ can be seen as the instance $\Cc_{\Gamma_1}=(X,\{0,1\},C_{\Gamma_1})$ of $\CSP(\Gamma_1)$ further restricted by a partial assignment $A=\{x_i=1-c \mid x_i\in Y\}$, for some $Y\subseteq X$ such that $C_{\Lambda_1}=C_{\Gamma_1}\cup A$. Note that in the combinatorial ideal 
    $\I_{\Cc_{\Lambda_1}}$ 
    all the variables in $Y$ are congruent to each other and we can work in a smaller polynomial ring. This suggests the following reduction from instance $\Cc_{\Lambda_1}$ to an instance of $\IMP_1(\Gamma_1)$: choose any $x_i\in Y$ and replace every occurrence of $x_j\in Y\setminus \{x_i\}$ in instance $\Cc_{\Gamma_1}=(X,\{0,1\},C_{\Gamma_1})$ with $x_i$ to get an instance $\Cc$ from $\CSP(\Gamma_1)$; consider the input polynomial $f=x_i-c$. Note that $(\Cc,f)$ forms a valid input for problem $\IMP_1(\Gamma_1)$ where we want to test if $f\in \I_\Cc$. We show that if we can test $f\not\in \I_\Cc$ in polynomial time then we can decide the satisfiability of $\Cc_{\Lambda_1}$ in polynomial time and the claim follows.

    As already observed in the proof of Lemma~\ref{th:Gd}, by the Strong Nullstellensatz~\eqref{eq:strong_nstz} and the radicality~\eqref{eq:ICradical} of $\I_\Cc$ (see~\eqref{eq:IC}),
    it follows that an equivalent way to solve the membership problem $p\in \I_\Cc$, for any given $p\in \Field[x_1,\ldots,x_n]$, is to answer the following question: Does it exist $\tilde{x}\in \{0,1\}^n$ such that $p(\tilde{x})\not=0 \wedge \tilde{x}\in \Variety{\I_\Cc}$?
    Note that the answer to this question is affirmative if and only if $p\not\in \Ideal{\Variety{\I_\Cc}}$ and therefore $p\not \in \I_\Cc$ by~\eqref{eq:strong_nstz} and~\eqref{eq:ICradical}. Vice versa, if it is negative then the following holds:
    %
    %\begin{align*}\label{eq:gboracle_no}
      $\forall\tilde{x}\in \{0,1\}^n  \left(p(\tilde{x})=0 \vee \tilde{x}\not\in \Variety{\I}\right)$,
    %\end{align*}
    %
    which implies that $p(x)=0$ for all $x\in \Variety{\I}$ and therefore $p\in \Ideal{\Variety{\I_\Cc}}=\I_\Cc$.
    With this in mind, note that $f\not\in\I_\Cc$ if and only if $\Cc_{\Lambda_1}$ is satisfiable and the claim follows.
    %
    %
    %%%%%%%%%%%%%%%%%%%%%%%%%%%%%%%%%%%%%%%%%%%%%%%%%%%%%%%%%%%%%%%%%%%%%%%%%%%%%%%%%%%%
    %%%%%%%%%%%%%%%%%%%%%%%%%%%%%%%%%%%%%%%%%%%%%%%%%%%%%%%%%%%%%%%%%%%%%%%%%%%%%%%%%%%%
    \item\label{case2} In the following we show that $\CSP(\Lambda_2)$ polynomial-time reduces to not-$\IMP_2(\Gamma_2)$. The latter implies that $\IMP_2(\Gamma_2)$ is coNP-complete as claimed. 
    
    The construction is similar to the previous case.  
    Let $\Cc_{\Lambda_2}=(X,\{0,1\},C_{\Lambda_2})$ be a given instance of $\CSP(\Lambda_2)$ and let $C_{\Gamma_2}\subseteq C_{\Lambda_2}$ be the maximal set of constraints from $C_{\Lambda_2}$ over the language $\Gamma_2$. 
    Instance $\Cc_{\Lambda_2}$ can be seen as the instance $\Cc_{\Gamma_2}=(X,\{0,1\},C_{\Gamma_2})$ of $\CSP(\Gamma_2)$ further restricted by a partial assignment $A=\{x_i=0 \mid x_i\in Y_0\}\cup \{x_i=1 \mid x_i\in Y_1\}$, for some $Y_0, Y_1\subseteq X$ with $Y_0\cap Y_1=\emptyset$ and such that $C_{\Lambda_2}=C_{\Gamma_2}\cup A$. Note that in the combinatorial ideal 
    $\I_{\Cc_{\Lambda_2}}$ 
    all the variables in $Y_0$ (or $Y_1$, respectively) are congruent to each other and we can work in a smaller polynomial ring. This suggests the following reduction from instance $\Cc_{\Lambda_2}$ to an instance of $\IMP_2(\Gamma_2)$: choose any $x_a\in Y_0$ (or $x_b\in Y_1$) and replace every occurrence of $x_j\in Y_0\setminus \{x_a\}$ and every occurrence of $x_k\in Y_1\setminus \{x_b\}$  in instance $\Cc_{\Gamma_2}=(X,\{0,1\},C_{\Gamma_2})$ with $x_a$ and $x_b$, respectively, to get an instance $\Cc_2=(X_2,\{0,1\},C_2)$ from $\CSP(\Gamma_2)$. Consider the following polynomial: % in normal form by $\{x_a^2-x_a,x_b^2-x_b\}$ (see Definition~\ref{def:reduction}):
    $f=x_b(1-x_a)$.
    Note that $f$ has degree 2 and $(\Cc_2,f)$ forms a valid input for problem $\IMP_2(\Gamma_2)$ where we want to test if $f\in \I_{\Cc_2}$. 
    Moreover the only solutions from $Sol(\Cc_2)$ that make $f$ identically zero are $x_a=x_b$ and $x_a=1\wedge x_b=0$.
    It follows that if we can test $f\not\in \I_\Cc$ in polynomial time then we can decide the satisfiability of $\Cc_{\Lambda_2}$ in polynomial time and the claim follows.
    %
    %%%%%%%%%%%%%%%%%%%%%%%%%%%%%%%%%%%%%%%%%%%%%%%%%%%%%%%%%%%%%%%%%%%%%%%%%%%%%%%%%%%%
    %%%%%%%%%%%%%%%%%%%%%%%%%%%%%%%%%%%%%%%%%%%%%%%%%%%%%%%%%%%%%%%%%%%%%%%%%%%%%%%%%%%%
    \item If the constraint language $\Gamma_2$ has the operation $\neg$ as a polymorphism then in the following we show that $\CSP(\Lambda_2)$ polynomial-time reduces to not-$\IMP_1(\Gamma_2)$. The latter implies that $\IMP_1(\Gamma_2)$ is coNP-complete as claimed. 
    
    Consider the above described instance $\Cc_2$ (see case \eqref{case2}) from $\CSP(\Gamma_2)$ and the polynomial $f=x_a-x_b$. Note that $(\Cc_2,f)$ forms a valid input for problem $\IMP_1(\Gamma_2)$ where we want to test if $f\in \I_{\Cc_2}$. If $f\not\in \I_{\Cc_2}$ then there is a solution $g:X_2\to \{0,1\}$ with $g(x_a)\not= g(x_b)$. If $g(x_a)=0 \wedge g(x_b)=1$ then the instance $\Cc_{\Lambda_2}$ from  $\CSP(\Lambda_2)$ has a solution. Otherwise suppose $g:X_2\to \{0,1\}$ satisfies $\Cc_2$ with $g(x_a)=1 \wedge g(x_b)=0$. Then the mapping $h: X_2\to \{0,1\}$ defined by $h(x)=\neg g(x)$ for all $x\in X_2$, satisfies the instance $\Cc_{\Lambda_2}$, since the constraint language $\Gamma_2$ has the operation $\neg$ as a polymorphism.
\end{enumerate}
\end{proof}
Let $\Gamma_2$ be a Boolean language with the solution space of every constraint closed under $1$ and $0$, but not simultaneously closed under any of $\{\Majority,\Minority,\Min,\Max,\neg\}$.
The only case that Lemma~\ref{th:necessary} does not cover is the complexity of $\IMP_1(\Gamma_2)$. Note that $\IMP_2(\Gamma_2)$ is coNP-complete but we do not know if $\IMP_1(\Gamma_2)$ is coNP-complete as well. However, we do not expect that $\IMP_1(\Gamma_2)$ is solvable in polynomial time as suggested by the following lemmas. 

\begin{lemma}
There is a Boolean constraint language $\Gamma_2$ with the solution space of every constraint closed under $1$ and $0$, but not simultaneously closed under any of $\{\Majority,\Minority,\Min,\Max,\neg\}$ operations, such that $\IMP_1(\Gamma_2)$ is coNP-complete.
\end{lemma}
\begin{proof}
Let $R_{\textsc{nae}}=\{0,1\}^3\setminus \{(0,0,0),(1,1,1)\}$ and let $\Lambda$ be the Boolean constraint language $\{R_{\textsc{nae}}\}$. By using Schaefer's Theorem~\ref{th:schaefer} we have that $\CSP(\Lambda)$ is NP-complete. %This problem is known under the name of \textsc{not-all-equal satisfiability} (\textsc{nae}).  

Let us define the following relations:% for $a,c\in\{0,1\}$:
\begin{align*}
 R&=\left \{ (z_1,z_2,z_3,0,1)\mid  (z_1,z_2,z_3)\in R_{\textsc{nae}} \right\}\cup \{(c,c,c,c,c)\mid c\in \{0,1\}\}.
\end{align*}
Note that $R$ is closed under both constant operations $1$ and $0$, but $R$ is not closed under any of $\{\Majority,\Minority,\Min,\Max,\neg\}$.

%The claim follows by showing that any given instance $\Ll$ of $\CSP(\Lambda)$ can be reduced to an instance of not-$\IMP_1(\{R\})$, namely to an instance $\Ga$ of $\CSP(\{R\})$ and a \emph{linear} polynomial $f$ such that if $f\not \in \I_{\Ga}$ then $\Ll$ admits a solution. 

Let $\Ll=(X,D,C)$ be an instance of $\CSP(\Lambda)$,
where $X=\{x_1,\ldots,x_n\}$ is a set of $n$ variables, $D=\{0,1\}$ and $C$  is a set of constraints over $\Lambda$ with variables from $X$. For any given instance $\Ll=(X,D,C)$ we build an instance $\Ga=(X\cup\{y_0,y_1\},D,C')$ of $\CSP(\{R\})$  with $C'=\{R(x_i,x_j,x_k,y_0,y_1)\mid R_{\textsc{nae}}(x_i,x_j,x_k)\in C\}$. 
The claim follows by observing that if $y_0-y_1\not \in \I_{\Ga}$ then $\Ll$ admits a solution.
\end{proof}

\begin{lemma}
Let $\Gamma_2$ be a Boolean constraint language with the solution space of every constraint closed under $1$ and $0$, but not simultaneously closed under any operation from $\{\Majority,\Minority,\Min,\Max,\neg\}$. If for every no-instance $(\Cc_2,f)$ of $\IMP_1(\Gamma_2)$, namely every instance $\Cc_2=(X_2,\{0,1\},C_2)$ of $\CSP(\Gamma_2)$ and a linear polynomial $f(X_2)\not\in \I_{\Cc_2}$, we can compute in polynomial time a feasible solution $g:X_2\to \{0,1\}$ such that $f(g(X_2))\not= 0$ with $g(X_2)\in Sol(\Cc_2)$, then P=NP. 
\end{lemma}
\begin{proof}
Consider the instance $\Cc_2=(X_2,\{0,1\},C_2)$ of $\CSP(\Gamma_2)$ described in case \eqref{case2} within the proof of Lemma~\ref{th:necessary} and the polynomial $f=x_a-x_b$.  Note that $(\Cc_2,f)$ forms a valid input for problem $\IMP_1(\Gamma_2)$ where we want to test if $f\in \I_{\Cc_2}$. If $f\not\in \I_{\Cc_2}$ then by the assumption we can efficiently compute a solution $g:X_2\to \{0,1\}$ with $g(x_a)\not= g(x_b)$. If $g(x_a)=0 \wedge g(x_b)=1$ then the instance $\Cc_{\Lambda_2}$ from  $\CSP(\Lambda_2)$ has a solution. 

Otherwise suppose $g:X_2\to \{0,1\}$ satisfies $\Cc_2$ with $g(x_a)=1 \wedge g(x_b)=0$. Then we build another instance of $\IMP_1(\Gamma_2)$ as follows. Create a copy of $\Cc_2$ say $\Cc_2'=(X_2',\{0,1\},C_2')$ with $X_2\cap X_2'=\emptyset$ ($C_2$ and $C_2'$ are the same but on disjoint sets of variables). In $\Cc_2'$, replace every occurrence of $x_a'$ with $x_b$ and every occurrence of $x_b'$ with $x_a$ and let $\Cc_2^*=(X_2^*,\{0,1\},C_2^*)$ denote the resulting instance. Consider $\Cc=(X_2\cup X_2^*,\{0,1\},C_2\cup C_2^*)$ and note that $\Cc$ is a valid instance of $\CSP(\Gamma_2)$.
Now note that if $f\not\in I_\Cc$ then there is a feasible solution $g:X_2\cup X_2^*\to \{0,1\}$ such that $g(x_a)\not=g(x_b)$. If $g(x_a)=1 \wedge g(x_b)=0$ then $\Cc_2'$ is satisfiable with $g(x_a')=0 \wedge g(x_b')=1$ and therefore the instance $\Cc_{\Lambda_2}$ from  $\CSP(\Lambda_2)$ has a solution (recall that we are assuming that $g:X_2\to \{0,1\}$ satisfies $\Cc_2$ with $g(x_a)=1 \wedge g(x_b)=0$). Symmetrically, if $g(x_a)=0 \wedge g(x_b)=1$ then the instance $\Cc_{\Lambda_2}$ from  $\CSP(\Lambda_2)$ has a solution ($\Cc_2'$ admits a feasible solution with $g(x_a')=1 \wedge g(x_b')=0$).
\end{proof}

%\newpage

%%%%%%%%%%%%%%%%%%%%%%%%%%%%%%%%%%%%%%%%%%%%%%%%%%%%%%%%%%%%%%%%%%%%%%%%%%%%%%%
\section{pp-definability and the Elimination of Variables}\label{sect:ppdef_zariski}
%%%%%%%%%%%%%%%%%%%%%%%%%%%%%%%%%%%%%%%%%%%%%%%%%%%%%%%%%%%%%%%%%%%%%%%%%%%%%%%
The key question on which the proof of Schaefer's Dichotomy Theorem centers is: For a given $\Gamma$, which relations are definable by \emph{existentially quantified $\Gamma$-formulas}? These existentially quantified formulas are known as \emph{pp-definable} relations:
\begin{definition}[\cite{Chen09}]\label{def:pp-def}
  A relation $R\subseteq D^k$ is \textbf{\emph{pp-definable}} (short for \emph{primitive positive definable}) from a constraint language $\Gamma$ if for some $m\geq 0$ there exists a finite conjunction $\mathscr{C}$ consisting of constraints over $\Gamma$ and equalities over $\{x_1,\ldots,x_m,x_{m+1},\ldots,x_{m+k}\}$ such that
  \begin{align}\label{eq:pp-def}
    R(x_{m+1},\ldots,x_{m+k}) & =\exists x_{1}\ldots \exists x_{m}\ \mathscr{C}.
  \end{align}
  That is, $R$ contains exactly those tuples of the form $(\phi(x_{m+1}),\ldots,\phi(x_{m+k}))$ where $\phi$ is an assignment that can be extended to a satisfying assignment of $\mathscr{C}$. We use $\langle \Gamma\rangle$ to denote the set of all relations that are pp-definable from $\Gamma$.
  %Let $\mathcal R$ be the set of tuples $(a_{m+1},\ldots,a_{m+k})\in D^k$ such that $R(a_{m+1},\ldots,a_{m+k})$ and let $\mathcal C$ be the set of tuples $(a_{1},\ldots,a_{m+k})\in D^{m+k}$ such that $C(a_{1},\ldots,a_{m+k})$.
\end{definition}
%
%Note that if $\Gamma'\subseteq \langle \Gamma\rangle$ then $\CSP(\Gamma')$ reduces to $CSP(\Gamma)$.
The notion of \emph{pp-definability} for relations permits a constraint language to ``simulate'' relations that might not be inside the constraint language.
The tractability of a constraint language $\Gamma$ is characterized by $\langle \Gamma\rangle$ and justifies focusing on the sets $\langle \Gamma\rangle$.
The set of relations $\langle \Gamma\rangle$ is in turn characterized by the polymorphisms of $\Gamma$.

In the following, we will explore the correspondence between pp-definability and \emph{elimination theory} in algebraic geometry (see e.g.~\cite{Cox:2015}).
It turns out that every pp-definable relation is equal to the smallest affine algebraic variety containing the set of solutions defined by the pp-definable relation, also known as the \emph{Zariski closure}. \GB bases can be used to compute the corresponding ideal (see Theorem~\ref{th:elim} and Lemma~\ref{th:pp-zariski}).
 This will allow us to construct a “dictionary” between geometry and algebra, whereby any statement about pp-definability (that can be seen as projection) can be translated into a statement about ideals (and conversely).\footnote{Note that the operation of taking image under a coordinate projection corresponds to existential quantification. % while that of taking fibre of coordinate projection corresponds to substituting parameter for variables.
 For example,
suppose that $R(x)\Leftrightarrow\exists y\ S(x,y)$. Let $\mathcal S$ be the set of pairs $(x,y)$ such that $S(x,y)$. Then $\mathcal{R}$ (i.e. the set of $x$ such that $R(x)$), is the projection of $\mathcal{S}$ into the first coordinate.}

%%%%%%%%%%%%%%%%%%%%%%%%%%%%%%%%%%%
\subsection{The Extension Theorem}\label{sect:extension}
%%%%%%%%%%%%%%%%%%%%%%%%%%%%%%%%%%%
We recall the notion of \emph{elimination ideal} (see \cite{Cox:2015}) from algebraic geometry.
\begin{definition}\label{def:elim}
  Given $\I=\langle p_1,\ldots,p_s\rangle \subseteq \Field[x_1,\ldots,x_{n}]$, for $0\leq m\leq n$, the $m$-th \emph{elimination ideal} $\I_m$ is the ideal of $\Field[x_{m+1},\ldots,x_{n}]$ defined by
  \begin{align*}
    \I_m &= \I \cap \Field[x_{m+1},\ldots,x_{n}].
  \end{align*}
\end{definition}
Thus, $\I_m$ consists of all consequences of $p_1=\cdots=p_s=0$ which eliminate the variables $x_1,\ldots, x_m$.
\begin{theorem}[\textbf{The Elimination Theorem}, \cite{Cox:2015}]\label{th:elim}
  Let $\I\subseteq \Field[x_1,\ldots ,x_n]$ be an ideal and let $G$ be a \GB basis of $\I$ with respect to lex order where $x_1>x_2>\cdots>x_n$. Then for every $0\leq m\leq n$, the set
  \begin{align*}
    G_m & = G \cap \Field[x_{m+1},\ldots ,x_n].
  \end{align*}
is a \GB basis of the $m$-th elimination ideal $\I_m$.
\end{theorem}

We will call a solution $\textbf{a}_m=(a_{m+1},\ldots, a_n)\in \Variety{\I_m}$ a \emph{partial solution} of the original system of equations. In general, it is not always possible to extend a partial solution $\textbf{a}_m$ (\emph{extension step}) to a complete solution in $\Variety{\I}$ (see e.g. \cite{Cox:2015}, Chapter 2). However when $\I_\Cc$ is defined as in~\eqref{eq:IC} then the extension step is always possible as shown by the following theorem.

\begin{theorem}[\textbf{The Extension Theorem}]\label{th:ext}
Let $\Cc$ be an instance of the $\CSP(\Gamma)$ and $\I$ defined as in~\eqref{eq:IC}.
For any $m\geq 0$ let $\I_m$ be the $m$-th elimination ideal (for any given ordering of the variables).
  Then, for any partial solution $\textbf{a}_m=(a_{m+1},\ldots, a_n)\in \Variety{\I_m}$ there exists an extension $c\in \Field^{m}$ such that $(c,b)\in \Variety{\I}$.
\end{theorem}
\begin{proof}
Note that if $\Variety{\I}=\emptyset$ then by the Weak-Nullstellensatz~\eqref{eq:weak_nstz} we have $1\in \I$ and therefore $1\in \I_m$ which implies by~\eqref{eq:weak_nstz} that $\Variety{\I_m}=\emptyset$. If the latter holds then the claim is vacuously true. Otherwise, $\Variety{\I_m}\not=\emptyset$ and $\Variety{\I}\not=\emptyset$. We assume this case in the following.

  By contradiction, assume that $\textbf{a}_m=(a_{m+1},\ldots, a_n)\in \Variety{\I_m}$ but $\textbf{a}_m$ does not extend to a feasible solution from $\Variety{\I}$. Then consider the following polynomial:
  \begin{align*}
   q(x_{m+1},\ldots,x_n)= \prod_{i\in \{m+1,\ldots,n\}}\ \prod_{j\in D\setminus\{a_i\}}(x_i-j).
  \end{align*}
  Note that
  \begin{align}\label{eq:ext_0}
  q(a_{m+1},\ldots, a_n)\not =0,
  \end{align}
  and any partial solution $(b_{m+1},\ldots, b_n)$ that can be extended to a feasible solution (there is one since we are assuming $\Variety{\I}\not=\emptyset$) would make $q(b_{m+1},\ldots, b_n) =0$. It follows that
  \begin{align}\label{eq:ext_1}
    q(x_{m+1},\ldots,x_n)&\in\Ideal{\Variety{\I}}\cap \Field[x_{m+1},\ldots,x_{n}].
  \end{align}
  By the definition of $\I$ and Theorem~\ref{th:nullstz} we have that
  \begin{align}\label{eq:ext_2}
    \Ideal{\Variety{\I}} &= \I.
  \end{align}
  By \eqref{eq:ext_1} and \eqref{eq:ext_2} it follows that
  \begin{align}\label{eq:ext_3}
    q(x_{m+1},\ldots,x_n)&\in \I_m,
  \end{align}
  and \eqref{eq:ext_0} implies that $\textbf{a}_m=(a_{m+1},\ldots, a_n)\not \in \Variety{\I_m}$, a contradiction.
\end{proof}

%%%%%%%%%%%%%%%%%%%%%%%%%%%%%%%%%%%%%%%%%%%%%%%%%%%%%%%%%%%%%
\subsection{pp-definability and Elimination Ideals}
%%%%%%%%%%%%%%%%%%%%%%%%%%%%%%%%%%%%%%%%%%%%%%%%%%%%%%%%%%%%%
In the following we analyze the relationship between pp-definability and Elimination Ideals. 
Consider any pp-definable relation $\mathcal{R}\subseteq D^k$  from a constraint language $\Gamma$ as given in \eqref{eq:pp-def}.  Then, for some $m\geq 0$ there exists a finite conjunction $\mathscr{C}$ consisting of constraints over $\Gamma$ and equalities over variables $\{x_1,\ldots,x_m,x_{m+1},\ldots,x_{m+k}\}$ such that $R(x_{m+1},\ldots,x_{m+k})  =\exists x_{1}\ldots \exists x_{m}\ \mathscr{C}$.

Then, we can find a set $P$ of polynomials (including domain  polynomials) $P=\{p_1,\ldots, p_s: p_i\in\Field[x_1,\ldots,x_m,x_{m+1},\ldots,x_{m+k}]\}\cup\{\prod_{j\in D}(x_i-j):i\in [k+m]\}$
such that $S=\Variety{P}$ is the set of satisfying assignments of $\mathscr{C}$ and $\Ideal{S}=\Ideal{P}$ (see Section~\ref{sect:idealCSP}).

%We show that $\mathcal R$ is Zariski closed.
%
%\begin{definition}
%  The \emph{\textbf{Zariski closure}} of a subset $S$ of affine space is the %smallest affine algebraic variety containing the set. If $S\subseteq \Field^k$, the %Zariski closure of $S$ is equal to $\Variety{\Ideal{S}}$.
%\end{definition}

%%%%%%%%%%%%%%%%%%%%%%%%%%%%%%%%%%%%%
\begin{lemma}\label{th:pp-zariski}
  %pp-definable relations $\mathcal R$ (as defined in \eqref{eq:pp-def}) are Zariski closed, i.e ${\mathcal R}= \Variety{\Ideal{\mathcal{R}}}$. In particular
  For any pp-definable relation $\mathcal{R}$ (as defined in \eqref{eq:pp-def}) we have that
  $\mathcal{R}=\Variety{\I_m}$, where $\I_m$ is the $m$-th elimination ideal of $\Ideal{S}$, i.e. $\I_m= \Ideal{S} \cap \Field[x_{m+1},\ldots,x_{m+k}]$.
\end{lemma}
%%%%%%%%%%%%%%%%%%%%%%%%%%%%%%%%%%%%
\begin{proof}
We define the projection of the affine variety $S=\Variety{P}$. We eliminate the first $m$ variables $x_1,\ldots,x_m$ by considering the \emph{projection map} 
$\pi_m: \Field^{m+k}\to \Field^k$, 
which sends $(a_1,\ldots,a_{m+k})$ to $(a_{m+1},\ldots,a_{m+k})$. By applying $\pi_m$ to $S$ we get $\pi_m(S)\subseteq \Field^k$. Note that the projection of $S$ corresponds to the pp-definable relation $\mathcal{R}$:
$\pi_m(S)=\mathcal{R}$.
We can relate  $\pi_m(S)$ to the $m$-th elimination ideal:
\begin{lemma}[See \cite{Cox:2015}, Sect. 2, Ch. 3]
\begin{align*}
  \pi_m(S) & \subseteq \Variety{\I_m}.
\end{align*}
\end{lemma}

Using the lemma above we can write $\pi_m(S)=\mathcal{R}$ as follows:
\begin{align*}
  \mathcal{R}=\pi_m(S) & =\{(a_{m+1},\ldots,a_{m+k})\in \Variety{\I_m}\mid \exists a_1,\dots,a_m\in \Field \text{ s.t. } (a_1,\ldots,a_{m+k})\in S\}.
\end{align*}
Note that $\pi_m(S)$ consists exactly of the partial solutions from $\Variety{\I_m}$ that extend to complete solutions. However, by the Extension Theorem~\ref{th:ext}, there is no partial solution from $\Variety{\I_m}$ that do not extend to complete solution. It follows that the pp-definable relation $\mathcal{R}$ in~\eqref{eq:pp-def} is exactly $\Variety{\I_m}$.
\end{proof}

By Lemma~\ref{th:pp-zariski} we see a realization of a well-known fact: quantifier-free definable sets are exactly the constructible sets in the Zariski topology (finite Boolean combinations of polynomial equations).
Indeed note that pp-definable relations are logically equivalent to a quantifier-free system of polynomial equations (from the elimination ideal). \GB basis (see Lemma~\ref{th:elim} and Lemma~\ref{th:pp-zariski}) is a way to compute this ``quantifier-free'' system of polynomials that are logically equivalent to pp-definable relations.

%%%%%%%%%%%%%%%%%%%%%%%%%%%%%%%%%%%%%%%%%%%%
\section{Conclusions and Future Directions}
%%%%%%%%%%%%%%%%%%%%%%%%%%%%%%%%%%%%%%%%%%%%
In this paper we identify restrictions on Boolean constraint languages $\Gamma$ which are sufficient and necessary to ensure the $\IMP_d(\Gamma)$ tractability for any fixed $d$. This result is obtained by combining techniques from both theory of $\CSP$s and computational algebraic geometry. We believe that it gives new insights into the applicability of \GB basis techniques to combinatorial problems.

Furthermore, this result can be applied for bounding the $\sos$ bit complexity and gives  necessary and sufficient conditions for the efficient computation of $d$-th Theta Body SDP relaxations of combinatorial ideals, identifying therefore the borderline of tractability for Boolean constraint language problems.

As it happened for $\CSP$ theory (see \cite{Bulatov17} and ~\cite{Zhuk17}), it would be very interesting to extend our dichotomy result to the finite domain case and understand which of the necessary/sufficient conditions for tractability of $\CSP(\Gamma)$ translate to the membership testing tractability.
%
%This restricted framework is still broad enough to include many problems from the class NP, yet it is narrow enough to potentially allow for complete classifications of all such $\IMP$s. %With this aim, we believe that the Interlacing Lemma~\ref{th:interlacing-lemma} will play an important role also for this generalization.

With this aim we mention a recent result \cite{BMmfcs20} by Bharathi and the author.
For the ternary domain, we consider problems constrained under the dual discriminator polymorphism and prove that we can find the reduced \GB basis of the corresponding combinatorial ideal in polynomial time. This ensures that we can check if any degree $d$ polynomial belongs to the combinatorial ideal or not in polynomial time, and provide membership proof if it does.
After the publication of \cite{BMmfcs20}, Bulatov and Rafiey have obtained new exciting results \cite{bulatov_rafiey20} by continuing this line of research and extending our results beyond Boolean domains in several ways. 
For example, in \cite{bulatov_rafiey20} it is shown that the $\IMP_d$ is solvable in polynomial time for any finite domain for problems constrained under the dual discriminator polymorphism. However, their approach  only works for the decision problem and does not allow one to find a \GB basis for the original problem (or a proof of ideal membership). On the other hand, although restricted to the ternary domain, in \cite{BMmfcs20} it is shown how to compute a \GB basis and therefore a proof of the IMP in polynomial time. It remains an interesting open problem to obtain a \GB basis for such problems over a general finite domain.
%

%Expanding the range of tractable $\IMP(\Gamma)$s. A number of candidates for such expansions are readily available from the existing results about the CSP. 
%
More in general, the study of $\CSP$-related Ideal Membership Problems is in its early stages and multiple directions are open. A very natural direction would be to expand the range of tractable $\IMP$s. Candidates for such expansions are readily available from the existing results about $\CSP$s. 
%
%In addition it would be nice to investigate also other kind of restrictions. Among other consequences, this would permit a better understanding of the bit complexity of $\sos$ proofs. 
%
%
\paragraph{Acknowledgments.}
I'm indebted with Andrei Bulatov for suggesting several useful comments and for spotting several gaps in the earlier version of the paper. 
I'm grateful to Standa \v{Z}ivn\'{y} for many stimulating discussions we had in Lugano and Oxford. I thank Arpitha Prasad Bharathi for several useful comments.

This research was supported by the Swiss National Science Foundation project 200020-169022 ``Lift and Project Methods for Machine Scheduling Through Theory and Experiments''.

{\small
\bibliographystyle{abbrv}
%\bibliography{ref}
\bibliography{ref}
}

\end{document}